\documentclass[a4paper,10pt,twoside]{article}
\pdfoutput=1

\usepackage{amsmath, amssymb, mathrsfs, amsthm, enumitem}
\usepackage[numbers]{natbib}
\usepackage[all]{xy}
\usepackage{graphicx}
\usepackage{wrapfig}
\usepackage{tikz}
\usetikzlibrary{arrows}
\usepackage{caption}
\usepackage{subcaption}
\usepackage{comment}
\usetikzlibrary{positioning}
\usepackage{hyperref}
\usepackage{bookmark}
\usepackage{tikz-cd}
\usepackage{float}
\usepackage[paperwidth=170mm, paperheight=240mm, hmarginratio=5:4, top=25mm]{geometry}

\usepackage{cleveref}
\usepackage{tabularx}
\usepackage{epigraph}

\newcommand{\Z}{\mathbb{Z}}

\newcommand{\Q}{\mathbb{Q}}
\newcommand{\C}{\mathbb{C}}

\renewcommand{\O}{\mathcal{O}}

\newcommand{\m}{\mathfrak{m}}

\newcommand{\bmt}{\begin{pmatrix}}
\newcommand{\emt}{\end{pmatrix}}
\newcommand{\bsm}{\left(\begin{smallmatrix}}
\newcommand{\esm}{\end{smallmatrix}\right)}

\renewcommand{\H}{\textnormal{H}}
\newcommand{\til}{\widetilde}

\DeclareMathOperator{\End}{End}
\DeclareMathOperator{\Gal}{Gal}
\DeclareMathOperator{\Hom}{Hom}

\DeclareMathOperator{\rk}{rk}

\DeclareMathOperator{\im}{im}

\DeclareMathOperator{\Spec}{Spec}

\DeclareMathOperator{\coker}{coker}

\DeclareMathOperator{\Sch}{\mathbf{Sch}}

\DeclareMathOperator{\Ab}{\mathbf{Ab}}

\DeclareMathOperator{\Pic}{Pic}

\newcommand{\ra}{\rightarrow}

\theoremstyle{definition}
\newtheorem{definition}{Definition}[section]

\newtheorem{remark}[definition]{Remark}
\newtheorem{example}[definition]{Example}
\newtheorem{construction}[definition]{Construction}

\theoremstyle{plain}
\newtheorem{proposition}[definition]{Proposition}
\newtheorem{lemma}[definition]{Lemma}
\newtheorem{theorem}[definition]{Theorem}
\newtheorem{corollary}[definition]{Corollary}

\theoremstyle{remark}

\renewcommand{\phi}{\varphi}

\author{Giulio Orecchia \\ email: \href{mailto:g.orecchia@math.leidenuniv.nl}{g.orecchia@math.leidenuniv.nl}}

\newcounter{nootje}
\newcommand{\beq}{\begin{equation}}
\newcommand{\eeq}{\end{equation}}
\newcommand{\beqs}{\begin{equation*}}
\newcommand{\eeqs}{\end{equation*}}
\numberwithin{equation}{section}

\title{A criterion for existence of N\'eron models of jacobians}

\begin{document}
\maketitle
 \begin{abstract}
N\'eron models of abelian varieties do not necessarily exist if the base $S$ has dimension higher than 1. We introduce a new condition, called toric additivity, on a family of smooth curves having nodal reduction over a normal crossing divisor $D\subset S$. The condition is necessary and sufficient for existence of a N\'eron model of the jacobian of the family; it depends only on the Betti numbers of the dual graphs of the fibres of the family, or on the toric ranks of the fibres of the jacobian.
 \end{abstract}

\section*{Introduction}
\subsection*{Toric additivity}
Consider an abelian variety $A$ over a number field $K$ with ring of integers $S$. In general, there may not exist an abelian scheme $\mathcal A/S$ extending $A$. However, it was proved by A. N\'eron and M. Raynaud that $A$ admits a canonical model $\mathcal N/S$, satisfying a number of good properties. Among these: it is a smooth, separated group scheme, and every $K$-valued point of $A$ extends uniquely to a section $S\ra \mathcal N$. Such a model is called a N\'eron model for $A$ (see \cref{defn_NM})

It is natural to ask whether an abelian variety always admits a N\'eron model if the base $S$ is, instead of a Dedekind scheme, a regular scheme of dimension higher than $1$.

In this paper, we focus on the case of jacobians of curves. We work over a regular base $S$, and consider a nodal curve $\mathcal C/S$, smooth over the complement of a normal crossing divisor $D\subset S$. We introduce a new condition on $\mathcal C/S$, smooth-local on $S$, called \textit{toric additivity} (\cref{def_ta_local}). Roughly, the family $\mathcal C/S$ is toric-additive at a geometric point $s$ of $S$ if the toric rank of the jacobian of the fibre over $s$ is equal to the sum of the toric ranks of the jacobian over the generic points of the components of the boundary divisor passing through $s$. More precisely, let $V$ be the spectrum of the \'etale local ring at $s$, $D_1,\ldots,D_n$ the irreducible components of $D\cap V$, $\zeta_i$ the generic point of $D_i$; there is an inequality (\ref{ineq_toric})
$$\dim T\leq \dim T_1+\ldots+\dim T_n$$
where $T$ is the maximal torus contained in $\Pic^0_{\mathcal C_s/s}$, and $T_i$ is the maximal torus contained in the fibre $\Pic^0_{\mathcal C_{\zeta_i}/\zeta_i}$. We say that $\mathcal C/S$ is \textit{toric-additive at $s$} if the inequality above is actually an equality.

The toric rank of the jacobian of a nodal curve (over an algebraically closed field) can also be interpreted as the first Betti number of the dual graph of the curve. Therefore, an equivalent definition of toric-additivity at $s$ is 
$$h_1(\Gamma)=h_1(\Gamma_1)+\ldots+h_1(\Gamma_n)$$ where $\Gamma$ is the dual graph of the nodal curve $\mathcal C_s/s$ and $\Gamma_i$ is the dual graph of the geometric fibre $\mathcal C_{\overline{\zeta}_i}$. 

Our main result is the following:
\begin{theorem}\label{mainthm}[\cref{main_thm_curves}]\mbox{}
Suppose $S$ is a regular scheme, $D$ a normal crossing divisor, $\mathcal C/S$ a nodal curve, smooth over $U=S\setminus D$. Write $J$ for the relative jacobian $\Pic^0_{\mathcal C_U/U}$.
\begin{itemize}
\item[a)] If $\mathcal C/S$ is toric-additive, then $J$ admits a N\'eron model over $S$.
\item[b)] If moreover $S$ is an excellent scheme, the converse is also true.
\end{itemize}  
\end{theorem}

The fact that the condition of toric additivity can be formulated merely in terms of the generalized jacobian $\Pic^0_{\mathcal C/S}$ -- the unique semi-abelian model of $\Pic^0_{\mathcal C_U/U}$, functorial with respect to $S$ -- is particularly convenient: it has the consequence that toric additivity (and, in view of \cref{mainthm}, the property of existence of a N\'eron model for the jacobian) is stable with respect to various types of base change. For example, the property is stable under any morphism $f\colon T\ra S$ such that $f^{-1}D$ is still a normal crossing divisor, and such that \'etale locally $f$ induces a bijection between the components of $f^{-1}D$ and of $D$ (see \cref{lemma_bc}).

Also, replacing the curve $\mathcal C/S$ by another nodal family $\mathcal C'/S$ via a blow-up of the total space $\mathcal C$ does not affect the property of toric additivity (\cref{TA_blowups}), since the blow-up induces the identity at the level of $\Pic^0$.

\subsection*{Relation with Holmes' alignment condition}
 The question of existence of N\'eron models of jacobians over higher dimensional bases was first raised in \citep{holmes}, where Holmes gave it a negative answer: for a regular base $S$, and a nodal curve $\mathcal C/S$ smooth over an open dense $U\subset S$, he showed that the jacobian $\Pic^0_{\mathcal C_U/U}$ does not always admit a N\'eron model. He actually showed that existence of N\'eron models of jacobians depends on a rather restrictive combinatorial condition, called \textit{alignment}, on the dual graphs of the fibres of $\mathcal C/S$ endowed with a certain labelling of the edges.

\begin{theorem}[\citep{holmes}, theorem 5.16, theorem 5.2]\label{ThmDavid}
Suppose $S$ is a regular scheme, $U\subseteq S$ an open dense, $\mathcal C/S$ a nodal curve, smooth over $U$. Let $J=\Pic^0_{\mathcal C_U/U}$ be the relative jacobian. 
\begin{itemize}
\item[i)] if $J$ admits a N\'eron model over $S$, then $\mathcal C/S$ is aligned;
\item[ii)] if moreover the total space $\mathcal C$ is regular, and $\mathcal C/S$ is aligned, then $J$ admits a N\'eron model over $S$.
\end{itemize}
\end{theorem}

In part ii), the hypothesis that the total space $\mathcal C$ is regular is essential. Holmes' work left open the question of whether the jacobians of non-regular, aligned nodal curves admit a N\'eron model. Even in the apparently tame case of \cref{example_nondisciplined} below it was not known what to expect.

\begin{example}\label{example_nondisciplined}
Let $k$ be a field, $S=\Spec k[[u,v]]$, $U=S\setminus \{uv=0\}$. Let $\mathcal E/S$ be the family of nodal curves of arithmetic genus one 
$$y^2-x^3-x^2-uv=0$$
which is smooth over $U$, and aligned. The total space $\mathcal E$ is not regular, as the point $x=0,y=0,u=0,v=0$ is singular, and there exists no nodal model $\mathcal E'/S$ of $\mathcal E_U$ with regular total space $\mathcal E'$.

\end{example}

Toric additivity is a stronger condition than alignment (\cref{ta_aligned}); and it turns out to be equivalent to it in the case where the total space $\mathcal C$ is regular (\cref{regular-aligned-TA}). It can be easily checked that the curve $\mathcal E/S$ of \cref{example_nondisciplined} is not toric-additive (as in this case $h_1(\Gamma)=1<1+1=h_1(\Gamma_1)+h_1(\Gamma_2)$), hence the jacobian of $\mathcal E_U$ does not admit a N\'eron model over $S$.

\subsection*{The proof of \cref{mainthm}}

The strategy of proof follows these lines: we show in \cref{disc->reg}, \cref{TA->disc} and \cref{NM->disc} that if the hypotheses of a) or b) are satisfied, there exists a blow-up $\mathcal C'\ra \mathcal C$, such that $\mathcal C'/S$ is still a nodal curve, smooth over $U$, and $\mathcal C'$ is \textit{regular}. \Cref{NM->disc} shows in particular that the curve of \cref{example_nondisciplined} does not admit a N\'eron model. As the properties of admitting a N\'eron model or of being toric-additive are not affected by desingularization, we have reduced to the case where the relative curve has regular total space. In this case, it can be shown that alignment and toric additivity are equivalent, and we apply \cref{ThmDavid} to reach the conclusion.

\subsection*{Outline}
In \cref{s1}, we define the objects with which we work - mainly nodal curves, their dual graphs, their jacobians. In \cref{section_TA}, we define toric additivity (\cref{def_ta_local}). In order to do this, we introduce the concept of purity map between character groups (\cref{character_map}). In \cref{sectionNM} we define N\'eron models (\cref{defn_NM}); after listing some of their properties, we look at their group of connected components and define a purity map (\cref{purity_phi}) also for connected components. In \cref{section_finale}, we recall the definition of alignment (\cref{def_aligned}) and compare it to toric additivity. Then we prove the key lemmas from which \cref{main_thm_curves} follows.

\subsection*{Acknowledgements}
This article stems from part of my PhD thesis; my warmest thank goes to my supervisor David Holmes for introducing me to this area of research and for the effort he put into supervising me.

I would also like to thank Hendrik W. Lenstra  for pointing out a significant error in a earlier version of my thesis; the corrections I applied improved the overall quality of this article a great deal.

Finally I thank Qing Liu, Bas Edixhoven, Raymond van Bommel and Thibault Poiret for useful discussions.
\newpage
\tableofcontents
\newpage
\section{Preliminaries}\label{s1}
\subsection{Nodal curves}

\begin{definition} A \textit{curve} $C$ over an algebraically closed field $k$ is a proper morphism of schemes $C\ra \Spec k$, such that $C$ is connected and its irreducible components have dimension $1$. A curve $C/k$ is called \textit{nodal} if for every non-smooth point $p\in C$ there is an isomorphism of $k$-algebras $\widehat{\O}_{C,p}\ra k[[x,y]]/xy$.

For a general base scheme $S$, a \textit{nodal curve} $f\colon\mathcal C\ra S$ is a proper, flat morphism of finite presentation, such that for each geometric point $s$ of $S$ the fibre $\mathcal C_s$ is a nodal curve.
\end{definition}

We will denote by $\mathcal C^{ns}$ the subset of $\mathcal C$ of points at which $f$ is not smooth. Seeing $\mathcal C^{ns}$ as the closed subscheme defined by the first Fitting ideal of $\Omega^1_{\mathcal C/S}$, we have for a nodal curve $\mathcal C/S$ that $\mathcal C^{ns}/S$ is finite, unramified and of finite presentation.

The local structure of nodal curves is described by the following lemma from \citep{holmes}.

\begin{lemma}[\citep{holmes}, Prop.2.5]\label{thickness_lemma}
Let $S$ be locally noetherian, $f\colon\mathcal C\ra S$ be nodal, and $p$ a geometric point of $\mathcal{C}^{ns}$ lying over a geometric point $s$ of $S$. We have:
\begin{itemize}
\item[i)] there is an isomorphism
$$\widehat \O^{sh}_{\mathcal C, p}\cong \frac{\widehat \O_{S,s}^{sh}[[x,y]]}{xy-\alpha}$$
for some element $\alpha$ in the maximal ideal of the completion $\widehat \O^{sh}_{S,s}$;
\item[ii)] the element $\alpha$ is in general not unique, but the ideal $(\alpha)\subset\widehat \O^{sh}_{S,s}$ is. Moreover, the ideal is the image in $\widehat \O^{sh}_{S,s}$ of a unique principal ideal $I\subset \O^{sh}_{S,s}$, which we call \textit{thickness} of $p$.
\end{itemize}
\end{lemma}

\begin{definition}\label{thickness_def}
In the hypothesis of \cref{thickness_lemma}, we call the ideal $I\subset \O^{sh}_{S,s}$ the \textit{thickness} of $p$. 
\end{definition}

We remark that, for $S$ regular at $s$, $\mathcal C$ is regular at $p$ if and only if the thickness of $p$ is generated by a regular parameter of the regular ring $\O^{sh}_{S,s}$.

\subsubsection{Split singularities}
Let $k$ be a field (not necessarily algebraically closed), $C/k$ a nodal curve, $n\colon C'\ra C$ its normalization. Following \citep[10.3.8]{Liu}, we say that $p\in C^{ns}$ is a \textit{split ordinary double point} if its preimage $n^{-1}(p)$ consists of $k$-valued points. This implies in particular that $p$ is $k$-valued. Moreover, if $p$ belongs to two or more components of $C$, then it belongs to exactly two components $Z_1,Z_2$; these are smooth at $p$ and meet transversally (\citep[10.3.11]{Liu}). We say that $C/k$ has \textit{split singularities} if every $p\in C^{ns}$ is a split ordinary double point. 

\begin{remark}
A nodal curve $C/k$ attains split singularities after a finite separable extension $k\ra k'$, by \citep[10.3.7 b)]{Liu}. \end{remark}

\begin{remark} A nodal curve with split singularities has irreducible components that are geometrically irreducible. Indeed, either $C/k$ is smooth, in which case it is geometrically connected and therefore geometrically irreducible; or every irreducible component of the normalization of $C$ contains a $k$-rational point and is therefore geometrically irreducible. 
\end{remark}

\begin{lemma}\label{thicknessZariski}
Let $\mathcal C\ra S$ be a nodal curve and $s\in S$ such that $\mathcal C_s$ has split singularities. Let $p$ be a geometric point of $\mathcal C_s$. Then the thickness $I$ of $p$ is generated by an element of the Zariski-local ring $\O_{S,s}$.
\end{lemma}
\begin{proof}
The morphism $f\colon \mathcal C^{ns}\ra S$ is finite unramified. Because $\mathcal C_s$ has split singularities, we see by \citep{stacks}\href{http://stacks.math.columbia.edu/tag/04DG}{TAG 04DG}, that there exists an open neighbourhood $U$ of $s$ such that $f^{-1}(U)\ra U$ is a disjoint union of closed immersions. In particular, $\mathcal C^{ns}\ra S$ is a closed immersion at $p$, and to it we can associate an ideal $J$ in the Zariski-local ring $\O_{S,s}$. We see (for example by \citep[proof of part 2 of Prop. 2.5]{holmes}) that $I$ is the image of $J$ in $\O_{S,s}^{sh}$; and moreover, since $\O_{S,s}\ra \O_{S,s}^{sh}$ is faithfully flat and $I$ is principal, $J$ is principal as well, which completes the proof.
\end{proof}

\subsection{The relative Picard scheme}
Let $S$ be a connected base scheme, and $\mathcal C\ra S$ a nodal curve of arithmetic genus $g$. We denote by $\Pic^0_{\mathcal C/S}$ the \textit{degree-zero relative Picard} functor; it is constructed as the fppf-sheaf associated to the functor 
\begin{eqnarray}
P^0_{\mathcal C/S}\colon \Sch/S & \ra & \Ab  \nonumber\\
T\ra S & \mapsto & \Pic^0(\mathcal C\times_ST) \nonumber
\end{eqnarray}
where by definition $\Pic^0(\mathcal C\times_ST)$ is the group of isomorphism classes of invertible sheaves $\mathcal L$ on $\mathcal C\times_ST$ such that, for every geometric point $t$ of $T$ and irreducible component $X$ of the fibre $\mathcal C_t$, $\deg \mathcal L_{|X}=0$.

It turns out that the degree-zero Picard functor $\Pic^0_{\mathcal C/S}$ of a nodal curve has an easy description if $\mathcal C/S$ admits a section. In this case, it is given by 
\begin{eqnarray}\Pic^0_{\mathcal C/S}\colon \Sch/S&\ra & \Ab  \nonumber\\
T\ra S & \mapsto & \frac{\Pic^0(\mathcal C\times_ST)}{\Pic(T)} \nonumber
\end{eqnarray}

If $\mathcal C/S$ is a smooth curve, it is well known that $\Pic^0_{\mathcal C/S}$ is represented by an abelian scheme, called the \textit{jacobian} of $\mathcal C/S$. If $\mathcal C/S$ is only nodal, then $\Pic^0_{\mathcal C/S}$ is represented by a semi-abelian scheme of relative dimension $g$ (\citep[9.4/1]{BLR}); in particular, for every point $s\in S$, there exists an exact sequence of fppf-sheaves
\begin{equation}\label{exseq_Pic} 0\ra T \ra \Pic^0_{\mathcal C_s/s}\ra B\ra 0 
\end{equation}
where $B=\Pic^0_{\widetilde{\mathcal C}/s}$ is the jacobian of the normalization of $\mathcal C_s$, hence an abelian variety, and $T/k(s)$ is the biggest subtorus of $\Pic^0_{\widetilde{\mathcal C}/s}$. 

We call $\mu(s):=\dim T$ and $\alpha(s):=\dim B$ respectively the \textit{toric rank} and \textit{abelian rank} of $\Pic^0_{\mathcal C_s/s}$. These numbers are stable under base field extension. Notice that $\Pic^0_{\mathcal C_s/s}$ is automatically geometrically connected. 

The function $\mu\colon S\ra \Z_{\geq 0}$ which associates to a point $s$ the toric rank of $\Pic^0_{\mathcal C_s/s}$ is upper semi-continuous (\citep{faltings1990degeneration}, Remark 2.4 Chapter I); defining analogously $\alpha\colon S\ra \Z_{\geq 0}$, we obtain that $\mu+\alpha$ is constant and equal to $g$.

\begin{definition}
Let $C/k$ be a nodal curve over a field, and let $\overline k$ be a separable closure. Let $T/k$ be the torus of the exact sequence \ref{exseq_Pic}. The \'etale $k$-group scheme $X:=\mathcal Hom_k(T,\mathbb G_m)$ is the datum of the free abelian group $X(\overline k)=\Hom_{\overline k}(T_{\overline k},\overline k^{\times})$ of rank $\mu$, endowed with the continuous action of $\Gal(\overline k|k)$. It is called the \textit{character group} of the torus $T$.
\end{definition}

If the torus $T$ is split, for example if $k$ is separably closed, then the action of $\Gal(\overline k|k)$ is trivial. In this case, $X$ is a constant group scheme and we will simply see it as an abstract free abelian group  of rank $\mu$. 

\subsection{Dual graphs of nodal curves}

We start by listing some graph-theoretical notions. In what follows, we will use the word \textit{graph} to refer to a finite, connected, undirected graph $G=(V,E)$. 

A \textit{path} on $G$ is a walk on $G$ in which all edges are distinct, and that never goes twice through the same vertex, except possibly for the first and last; a \textit{cycle} is a path that starts and ends at the same vertex. A \textit{loop} is a cycle consisting of only one edge.

Let now $C$ be a curve with split singularities over a field $k$. We define the \textit{dual graph} of $C$ as in \citep[10.3.17]{Liu}: it is the graph $\Gamma=(V,E)$ with $V=\{\mbox{irreducible components of }C\}$, $E=\{\mbox{singular points of }C\}$; the extremal vertices of an edge $p$ are the components containing $p$, which are indeed either one or two. 

The following result gives a graph theoretic interpretation of the character group of a curve and its rank.

\begin{lemma}\label{lemma_split_torus}
Let $C/k$ be a nodal curve with split singularities.

\begin{itemize}
\item[i)] The torus $T$ of the exact sequence \ref{exseq_Pic} is split.
\item[ii)] Let $\Gamma$ the dual graph of $C$. The character group $X$ of $T$ is canonically identified with the first homology group $H_1(\Gamma,\Z).$ In particular the first Betti number $h_1(\Gamma,\Z)$ is equal to $\mu=\rk X=\dim T=\mbox{toric rank of }\Pic^0_{C/k}.$
\end{itemize}
\end{lemma}

\begin{proof}
The normalization morphism and the structure morphisms fit into a commutative diagram
\begin{center}
\begin{tikzcd}
\widetilde C \arrow[r, "\pi"] \arrow[rd, "\widetilde p"] & C  \arrow[d, "p"] \\
& \Spec k 
\end{tikzcd}
\end{center}

Consider the exact sequence 
\begin{equation}\label{ex_seq_push}1\ra \O_{C}^{\times}\ra \pi_*\O_{\widetilde C}^{\times} \ra \frac{\pi_*\O_{\widetilde C}^{\times}}{\O_{C}^{\times}}\ra 1.
\end{equation}
It gives a long exact sequence
$$ p_*\O_{C}^{\times}\ra \widetilde p_*\O_{\widetilde C}^{\times}\ra p_*\frac{\pi_*\O_{\widetilde C}^{\times}}{\O_{C}^{\times}}\ra R^1p_*\O_{C}^{\times}\ra R^1p_*(\pi_*\O_{\widetilde C}^{\times})
$$

We write $V=\{v_1,\ldots,v_n\}$ for the set of components of $\widetilde C$ (the set of vertices of $\Gamma$) and $E=\{e_1,\ldots,e_r\}$ for the set of singular points of $C$ (the set of edges of $\Gamma$). The morphism $\O_C\ra \pi_*\O_{\widetilde C}$ is an isomorphism on the smooth locus and is given by the diagonal morphism $k\ra k\oplus k$ at the singular points because $C$ has split singularities. We make a choice of orientation of the edges of the dual graph, which allows us to identify the quotient of the diagonal morphism $k\ra k\oplus k$ above with the morphism $k\oplus k\ra k$, $(x,y)\ra x-y$. Hence the term $p_*\frac{\pi_*\O_{\widetilde C}^{\times}}{\O_{C}^{\times}}$ is identified with $\bigoplus_E \mathbb G_{m,k}$. Moreover, $R^1p_*\O_{C}^{\times}=\Pic^0_{C/k}$ and the natural morphism $R^1p_*(\pi_*\O_{\widetilde C}^{\times})\ra R^1\widetilde p_*\O_{\widetilde C}^{\times}=\Pic^0_{\widetilde C/k}$ is injective. We obtain an exact sequence 
$$1\ra \mathbb G_{m,k}\ra \bigoplus_{V}\mathbb G_{m,k}\ra \bigoplus_{E}\mathbb G_{m,k}\ra  \Pic^0_{C/k}\ra \Pic^0_{\widetilde C/k}.$$
The first map is the diagonal; the second is given by tensoring with $\mathbb G_{m,k}$ the linear map $M\colon k^V\ra k^E$ given by the incidence matrix of the oriented graph $\Gamma$. It follows that the kernel of $\Pic^0_{C/k}\ra \Pic^0_{\widetilde C/k}$ is the split torus $\coker M\otimes_k\mathbb G_{m,k}$. Notice that its rank is indeed $r-n+1=h_1(\Gamma,\Z)$. This proves i) and the second part of ii). For the remaining part of ii), apply $\Hom_k(\_\_,\mathbb G_{m,k})$ to the exact sequence 
$$\bigoplus_{V}\mathbb G_{m,k}\ra \bigoplus_{E}\mathbb G_{m,k}\ra T\ra 0$$
to obtain an exact sequence
$$0\ra X\ra \Z^E\ra \Z^V$$
the rightmost map being given by the transpose of the incidence matrix. This identifies $X$ with $H_1(\Gamma,\Z)$. 
\end{proof}

\subsubsection{Labelled dual graphs}\label{labelled}
Given a nodal curve $f\colon \mathcal C\ra S$ and a point $s$ of $S$ such that $\mathcal C_s$ has split singularities, we write $\Gamma_s=(V_s,E_s)$ for the dual graph associated to the fibre $\mathcal C_s$. Using the notation of \citep{holmes}, we write $L_s$ for the monoid of principal ideals of the (Zariski-)local ring $\O_{S,s}$; then we let $l_s\colon E_s\ra L_s$ be the function associating to each edge of $\Gamma_s$ the thickness of the corresponding singular point of $\mathcal C_s$ (which indeed is an ideal of $\O_{S,s}$, by \cref{thicknessZariski}). 

\begin{definition}\label{def_labelled_graph}The pair $(\Gamma_s,l_s)$ constructed above is called \textit{labelled graph} of $\mathcal C\ra S$ at the point $s$.
\end{definition}
\subsection{The generization map}
Let $\mathcal C\ra S$ be a nodal curve over a regular, strictly local scheme. Let $Z$ be a regular closed subscheme of $S$, defined by an ideal $\mathfrak a\subset \O(S)$ and with generic point $\zeta$, and let $s$ be the closed point of $S$. Our aim is to compare the fibres $\mathcal C_{\zeta}$ and $\mathcal C_s$, in terms of their graphs and character groups.

\begin{lemma}\label{split_sing_sh}
The fibre $\mathcal C_{\zeta}$ has split singularities. 
\end{lemma}
\begin{proof}
We drop the name $Z$ and simply assume that $\zeta$ is the generic point of $S$.
The non-smooth locus $\mathcal C^{ns}$ is finite unramified over $S$, hence a disjoint union of closed subschemes of $S$. Let $X\subseteq \mathcal C^{ns}$ be the part consisting of sections $S\ra \mathcal C$. Notice that on the generic fibre $X_{\zeta}=\mathcal C^{ns}_{\zeta}$. 

We claim that the open subscheme $\mathcal C\setminus X$ is normal. We will show it by using Serre's criterion for normality (\citep[8.2.23]{Liu}). First, as $X$ has been removed, $\mathcal C\setminus X$ is regular at its points of codimension $1$. Condition $S_2$ follows from the fact that $\mathcal C\setminus X$ is locally complete intersection over a regular, noetherian base, hence Cohen-Macaulay by \citep[8.2.18]{Liu}. This proves the claim. 

Our next claim is that the normalization $\pi\colon \mathcal C'\ra \mathcal C$ is finite and unramified. Since these are properties fpqc-local on the target, and since we already know that $\pi$ induces an isomorphism over $\mathcal C\setminus X$, it is enough to check the claim over the completion of the strict henselization of geometric points of $X$. Let $x$ be such a geometric point lying over a geometric point $s$ in $S$. Then $x$ has thickness zero, hence $\widehat{\O}_{\mathcal C,x}^{sh}\cong \frac{\widehat{\O}_{S,s}^{sh}[[u,v]]}{uv}.$
The normalization morphism 
$$\Spec \widehat{\O}_{S,s}^{sh}[[u]]\times \widehat{\O}_{S,s}^{sh}[[v]]\ra \Spec \frac{\widehat{\O}_{S,s}^{sh}[[u,v]]}{uv}$$
is indeed finite and unramified, which proves the claim.

Now, let $Y$ be the preimage of $X$ via $\pi\colon \mathcal C'\ra \mathcal C$. We have that $Y$ is finite, unramified over $X$, and in particular finite \'etale over $S$. Hence $Y$ is a disjoint union of sections $S\ra \mathcal C'$. The restriction of $\pi$ to the generic fibre, $\pi_{\eta}\colon \mathcal C'_{\eta}\ra\mathcal C_{\eta}$, is a normalization morphism, and we see that the preimage $Y_{\eta}$ of $X_{\eta}=(\mathcal C_{\eta})^{ns}$ consists of $k(\eta)$-valued points, as we wished to show.

\end{proof}

\begin{construction}\label{contraction_graph}
We can therefore associate a labelled graph (\cref{def_labelled_graph}) $\Gamma_{\zeta}=(V_{\zeta},E_{\zeta})$ to $\mathcal C_{\zeta}$. Every singular point of $\mathcal C_{\zeta}$ specializes to a singular point of $\mathcal C_s$ with thickness contained in the ideal $\mathfrak a$ of $Z$ in $\O_{S,s}$. So if $(\Gamma_s=(V_s,E_s),l_s)$ is the labelled graph of $\mathcal C_s$, we have an inclusion $E_{\zeta}\subset E_s$; and the labelled graph $\Gamma_{\zeta}$ is obtained from the labelled graph $\Gamma_s$ of the closed fibre by:
\begin{enumerate}
\item contracting all edges labelled by an ideal of $\O_{S,s}$ not contained in $\mathfrak a$; 
\item labelling edges in $\Gamma_{\zeta}$ by the image via the inclusion $\O_{S,s}\ra\O_{S,\zeta}$ of the label of the corresponding edge in $\Gamma_s$.
\end{enumerate}
\end{construction}

We also obtain a natural commutative diagram
\begin{equation}\label{diagram_bho}
\begin{tikzcd}
\Z^{E_s} \arrow[r]\arrow[d] & \Z^{V_s} \arrow[d] \\
\Z^{E_{\zeta}} \arrow[r] & \Z^{V_{\zeta}}
\end{tikzcd}
\end{equation}
The horizontal maps are the usual boundary maps, given by the incidence matrices of $\Gamma_s$ and $\Gamma_{\zeta}$. The left vertical map is induced by the inclusion $E_{\zeta}\subseteq E_s$. For the right vertical map, notice that we have a natural surjective map $\mbox{gen}_V\colon V_{s}\ra V_{\zeta}$; then, the image of a vertex labelling $\phi$ via $\Z^{V_s}\ra \Z^{V_{\zeta}}$ is the vertex labelling $\phi'$ with $\phi'(v)=\sum_{w\in \mbox{gen}_V^{-1}(v)} \phi(w).$

Taking kernels of the horizontal maps in diagram \ref{diagram_bho} we obtain a \textit{generization map} on homology 
\begin{equation}\label{generization_homology}
H_1(\Gamma_s,\Z)\ra H_1(\Gamma_{\zeta},\Z)
\end{equation}

Consider now the semi-abelian scheme $\Pic^0_{\mathcal C/S}$, and its restriction $P:=\Pic^0_{\mathcal C_Z/Z}$ to $Z$. Let $T'$ be the maximal torus contained in $P_{\zeta}$, which is split by \cref{lemma_split_torus}, and $T$ be the maximal torus contained in $P_{s}$. Let $X'$ and $X$ be the respective character groups. The closure of $T'$ inside $P$ is a split subtorus $\mathcal T'\subseteq P$ by \citep[Prop. 2.9 Chapter 1]{faltings1990degeneration}. In particular, the restriction $\mathcal T'_s$ is a subtorus of $T$, having character group $X'$. The inclusion $\mathcal T'_s\subseteq T$ induces a surjective homomorphism of free abelian groups 
\begin{equation}\label{specialization_tori}
X\ra X'
\end{equation}
called as well \textit{generization map} of characters groups.
Remember that by \cref{lemma_split_torus}, $X=H_1(\Gamma_s,\Z)$ and $X'=H_1(\Gamma_{\zeta},\Z)$. 

\begin{lemma}
The two generization maps $H_1(\Gamma_s)\ra H_1(\Gamma_{\zeta})$ and $X\ra X'$ are the same map.
\end{lemma}
\begin{proof}
We write $\mathcal C/Z$ for the restriction of the curve to $Z$. Let $\pi\colon \widetilde {\mathcal C}\ra \mathcal C$ be the normalization. The restriction of $\pi$ to the generic fibre $\pi_{\zeta}\colon \widetilde{\mathcal C}_{\zeta}\ra \mathcal C_{\zeta}$ is also a normalization morphism. On the other hand the restriction to $s$, $\pi_s\colon \widetilde{\mathcal C}_s\ra \mathcal C_s$ is  only the partial normalization at the subset of singular points $E_{\zeta}\subseteq E_s$. Normalizing the remaining nodes we find the normalization $\pi^{\circ} \colon  C^{\circ}\ra \mathcal C_s$. We call $p\colon \mathcal C_s\ra \Spec k$, $q\colon \widetilde{\mathcal C}_s\ra \Spec k$, $r\colon C^{\circ}\ra \Spec k$ the three structure morphisms.

As in \cref{ex_seq_push}, the torus $T'$ fits into an exact sequence
$$0\ra \mathbb G_{m,\zeta}\ra \bigoplus_{V_{\zeta}}\mathbb G_{m,\zeta}\ra \bigoplus_{E_{\zeta}}\mathbb G_{m,\zeta}\ra T'\ra 0$$
which prolongs to an exact sequence 
$$0\ra \mathbb G_{m,Z}\ra \bigoplus_{V_{\zeta}}\mathbb G_{m,Z}\ra \bigoplus_{E_{\zeta}}\mathbb G_{m,Z}\ra \mathcal T'\ra 0$$

We have a commutative diagram with exact rows

$$
\begin{tikzcd}
0 \ar[r] & \O_{\mathcal C_s}^{\times}\ar[r]\ar[d] & \left(\pi_*\O_{\widetilde{\mathcal C}}^{\times}\right)_s \ar[r]\ar[d] & \left(\frac{\pi_*\O_{\widetilde{\mathcal C}}^{\times}}{\O_{\mathcal C}^{\times}}\right)_s \ar[r]\ar[d] & 0\\
0 \ar[r] & \O_{\mathcal C_s}^{\times}\ar[r] & \pi^{\circ}_*\O_{C^{\circ}}^{\times} \ar[r] & \frac{\pi^{\circ}_*\O_{C^{\circ}}^{\times}}{\O_{\mathcal C_s}^{\times}} \ar[r] & 0.
\end{tikzcd}
$$

Applying the functor $p_*$, we obtain a commutative diagram with exact rows

\begin{equation}
\begin{tikzcd}
0 \ar[r] & \mathbb G_{m,s}\ar[r]\ar[d] & \bigoplus_{V_{\zeta}}\mathbb G_{m,s} \ar[r]\ar[d] & \bigoplus_{E_{\zeta}}\mathbb G_{m,s} \ar[r]\ar[d] & \mathcal T'_s \ar[r]\ar[d] & 0\\
0 \ar[r] & \mathbb G_{m,s}\ar[r] & \bigoplus_{V_s}\mathbb G_{m,s} \ar[r] & \bigoplus_{E_s}\mathbb G_{m,s} \ar[r] & T \ar[r] & 0.
\end{tikzcd}
\end{equation}

The leftmost vertical map is the identity; the next vertical map is the natural map $q_*\O_{\widetilde{\mathcal C}_s}^{\times}\ra r_*\O_{\mathcal C_s^*}^{\times}$. Therefore it is simply the map induced by $\mbox{gen}\colon V_s\ra V_{\zeta}$. With a similar reasoning one deduces that the next vertical map is obtained by the inclusion $E_{\zeta}\ra E_s$. Finally, the rightmost vertical map is the inclusion of split tori found previously. 

Now, applying $\Hom_s(\_\_,\mathbb G_{m,s})$, we obtain the commutative diagram with exact rows
\begin{equation}
\begin{tikzcd}
0 \arrow[r] & X \arrow[r]\arrow[d] & \Z^{E_s} \arrow[r]\arrow[d] & \Z^{V_s} \arrow[d] \\
0\arrow[r] & X'\arrow[r] & \Z^{E_{\zeta}} \arrow[r] & \Z^{V_{\zeta}}
\end{tikzcd}
\end{equation}
where the first vertical map is the generization map $X\ra X'$ and the rightmost square is diagram \ref{diagram_bho}. 
\end{proof}

\section{Toric additivity}\label{section_TA}

\subsection{The purity map}
We consider a regular, strictly local base scheme $S$ with a normal crossing divisor $D=D_1\cup D_2\cup\ldots\cup D_n\subset S$ for some $n\geq 0$. We let $t_1,\ldots,t_n\in \O_{S,s}$ be the functions cutting out the components $D_1,\ldots,D_n$. We also let $\mathcal C/S$ be a nodal curve such that the base change $C_U/U$ to $U=S\setminus D$ is smooth. 

Write $\zeta_1,\ldots,\zeta_n$ for the generic points of the irreducible components $D_1,\ldots,D_n$, and $C_1,\ldots, C_n$ for the fibres over each of them. Each $C_i$ has split singularities, hence by \cref{lemma_split_torus} we can attach to it a dual graph $\Gamma_i$, a split torus $T_i\subseteq \Pic^0_{C_i}$ of dimension $\mu_i=h_1(\Gamma_i,\Z)$, and a free abelian group $X_i=H_1(\Gamma_i,\Z)$ of rank $\mu_i$. Similarly, we let $s\in S$ be the closed point, with residue field $k$, and $C/k$ be the fibre, with dual graph $\Gamma$, (split) torus $T$ of dimension $\mu=h_1(\Gamma,\Z)$, character group $X=H_1(\Gamma,\Z)$ of rank $\mu$. 

For every $i=1,\ldots,n$ there is a generization map \ref{specialization_tori} of character groups $X\ra X_i$; It is natural to define the following:

\begin{definition}\label{character_map}
We call \textit{purity map} the group homomorphism
\begin{equation}\label{purity_char}
X\ra X_1\oplus X_2\oplus\ldots X_n
\end{equation}
induced by the generization maps $X\ra X_i$.
It may be seen as a homomorphism of homology groups
\begin{equation}\label{purity_betti}
H_1(\Gamma,\Z)\ra H_1(\Gamma_1,\Z)\oplus\H_1(\Gamma_2,\Z)\oplus\ldots\oplus H_1(\Gamma_n,\Z).
\end{equation}
\end{definition}

\begin{lemma}\label{injective_purity}
The purity map $X\ra X_1\oplus\ldots\oplus X_n$ is injective.
\end{lemma}
\begin{proof}
Every edge of $\Gamma$ is labelled by a principal ideal generated by a product of powers of $t_1,t_2,\ldots,t_n$. Hence every edge is preserved in at least one contraction $\Gamma_i$, as one can see from \cref{contraction_graph}. It immediately follows that the canonical morphism  $H_1(\Gamma,\Z)\ra H_1(\Gamma_1,\Z)\oplus\ldots\oplus H_1(\Gamma_n,\Z)$ is injective. 
\end{proof}

\begin{corollary}\label{inequality}
We have the inequalities of toric ranks
\begin{equation}\label{ineq_toric}
\mu\leq \mu_1+\mu_2+\ldots+\mu_n.
\end{equation}
and of Betti numbers
\begin{equation}\label{ineq_betti}
h_1(\Gamma,\Z)\leq h_1(\Gamma_1,\Z)+h_1(\Gamma_2,\Z)+\ldots+h_1(\Gamma_n,\Z).
\end{equation}
\end{corollary}
\begin{proof}
Follows by taking ranks in the injective homomorphism $X\ra X_1\oplus X_2\oplus \ldots X_n$.
\end{proof}

\begin{lemma}\label{tf_coker}
The purity map $X\ra X_1\oplus X_2\oplus\ldots\oplus X_n$ has torsion-free cokernel.
\end{lemma}

\begin{proof}
We have a commutative diagram with injective arrows
\begin{center}
\begin{tikzcd}
H_1(\Gamma,\Z)\arrow[hookrightarrow, r, "\alpha"]\arrow[hookrightarrow, d, "\gamma"] & \bigoplus H_1(\Gamma_i,\Z) \arrow[hookrightarrow, d, "\delta"] \\
\Z^{E}\arrow[hookrightarrow, r, "\beta"] & \bigoplus \Z^{E_i} 
\end{tikzcd}
\end{center}
and we want to show that $\coker \alpha$ is torsion-free. Clearly $\coker \beta$ is torsion-free. Moreover, $\coker \gamma $ is contained in $\Z^V$ and is therefore torsion-free. By the Snake Lemma, the kernel of the induced map $\coker \alpha \ra \coker \beta$ is contained in $\coker \gamma$, hence is torsion-free. It follows that $\coker \alpha$ is torsion-free.
\end{proof}
\begin{corollary}
The purity map $X\ra X_1\oplus X_2\oplus\ldots\oplus X_n$ is an isomorphism if and only if the rank of $X$ is equal to the rank of $X_1\oplus\ldots\oplus X_n$, that is, if and only if $\mu=\mu_1+\ldots +\mu_n$.
\end{corollary}
\begin{proof}
Bijecitivity of the purity map implies equality of ranks. If, on the other hand, $\mu=\mu_1 +\ldots +\mu_n$, then the injective map $X\ra X_1\oplus\ldots\oplus X_n$ has a finite cokernel $F$. By \cref{tf_coker}, $F$ is also torsion-free, hence it is zero.
\end{proof}

\subsection{Definition of toric-additive curve}
\begin{definition}[Local definition of toric additivity]\label{def_ta_local}
We say that the nodal curve $\mathcal C/S$ is \textit{toric-additive} if the following equivalent conditions are satisfied:
\begin{itemize}
\item[i)] the purity map (\cref{character_map}) of character groups $$X\ra X_1\oplus X_2\oplus \ldots X_n$$ is an isomorphism;
\item[ii)] the purity map of homology groups 
$$H_1(\Gamma,\Z)\ra H_1(\Gamma_1,\Z)\oplus H_1(\Gamma_2,\Z)\oplus\ldots\oplus H_1(\Gamma_n,\Z)$$ is an isomorphism;
\item[iii)] $\mu=\mu_1+\mu_2+\ldots+\mu_n$;
\item[iv)] $h_1(\Gamma,\Z)=h_1(\Gamma_1,\Z)+h_1(\Gamma_2,\Z)+\ldots+h_1(\Gamma_n,\Z)$.

\end{itemize}
\end{definition}

\begin{remark}\label{remark_ta}
If $n=1$, then $\mathcal C/S$ is automatically toric-additive. Indeed, $\mu_1=\rk X_1\leq \rk X=\mu$ (by upper semi-continuity of the toric rank, or because the graph $\Gamma_1$ is a contraction of $\Gamma$), but at the same time \cref{ineq_toric} tells us $\mu\leq \mu_1$. 

In the trivial case $n=0$, $\mathcal C/S$ is smooth, and is toric-additive, as the character group $X$ is zero and the empty sum is zero.

\end{remark}
\begin{lemma}\label{ta_lemmino}
Suppose that $\mathcal C/S$ is toric-additive and that $t$ is a geometric point of $S$. Let $T=\Spec \O_{S,t}^{sh}$ be the strict henselization and $\mathcal C_T/T$ the base change. Then $\mathcal C_T/T$ is still toric-additive.
\end{lemma}

\begin{proof}
Without loss of generality, we assume that $t$ is a geometric point of $D_1,D_2,\ldots,D_m$, for some integer $0\leq m\leq n$. We let $\Gamma'$ be the graph of the fibre $\mathcal C_t$ and $X'$ the associated group of characters. We obtain a purity map $X'\ra X_1\oplus\ldots\oplus X_m$. The isomorphism $X\ra X_1\oplus\ldots\oplus X_n$ factors as
$$X\ra X'\oplus X_{m+1}\oplus X_{m+2}\oplus\ldots\oplus X_n\ra X_1\oplus X_2\oplus\ldots\oplus X_n$$
where the first arrow is a sum of generization maps, and the second arrow is the direct sum of the purity map $X'\ra X_1\oplus X_2\oplus\ldots\oplus X_m$ and of the identity morphisms $X_i\ra X_i$, $m+1\leq i\leq n$. 
Hence the second arrow is both injective and surjective, and it follows that $X'\ra X_1\oplus \ldots\oplus X_m$ is an isomorphism.

\end{proof}

We now drop the hypothesis that $S$ is strictly local, and give a global definition of toric additivity. Let $S$ be regular, with a normal crossing divisor $D\subset S$ and let $\mathcal C/S$ be a nodal curve, smooth over $U=S\setminus D$.

\begin{definition}[Global definition of toric additivity]
We say that $\mathcal C/S$ is toric-additive at a geometric point $s$ of $S$ if the base change to the strict henselization $\Spec \O^{sh}_{S,s}$ is toric-additive as in \cref{def_ta_local}.

We say that $\mathcal C/S$ is toric-additive if it is toric-additive at all geometric points $s$ of $S$. In view of \cref{ta_lemmino}, this definition is consistent with \cref{def_ta_local}.
\end{definition}

\subsection{Base change properties}

The functoriality of $\Pic^0_{\mathcal C/S}$ allows us to show that toric additivity behaves well with respect to a quite general class of base change.

\begin{lemma}\label{lemma_bc}
Let $S$ be a regular scheme, $D\subset S$ a normal crossing divisor and $\mathcal C/S$ a nodal curve smooth over $U=S\setminus D$. Let $T$ be another regular scheme, and $f\colon T\ra S$ a morphism such that $E:=(D\times_ST)_{red}\subset T$ is a normal crossing divisor. Let $t$ be a geometric point of $T$, lying over some geometric point $s$ of $S$, and let $f'\colon T'\ra S'$ be the morphism between the strict henselizations. Suppose that $f'$ satisfies the following assumption:

$(\star)$ For every irreducible component $E_i$ of $E\cap T'$, the image via $f'$ is a irreducible component $D_j$ of $D\cap S'$, and the induced function $\{E_i\}_i\ra \{D_j\}_j$ is a bijection between the set of irreducible components of $E\cap T'$ and the set of irreducible components of $D\cap S'$.

Then $\mathcal C/S$ is toric-additive at $s$ if and only if the base-change $\mathcal C_T/T$ is toric-additive at $t$.
\end{lemma}

\begin{proof}
The morphism $f'\colon T'\ra S'$ induces a bijection between the generic points $\xi_1,\ldots,\xi_n$ of $E\cap T'$ and the generic points $\zeta_1,\ldots,\zeta_n$ of $D\cap S'$. Clearly $\mu(t)=\mu(s)$ and $\mu(\xi_i)=\mu(\zeta_i)$, and the statement follows.
\end{proof}

\begin{corollary}\label{TA_smooth}
Toric additivity is local on the target for the smooth topology.
\end{corollary}

\begin{proof}
Let $f'\colon T'\ra S'$ be the induced morphism between strict henselizations and let $D_i$ be any component of $D\times_SS'$. The induced morphism from $E_i:=D_i\times_{S'}T'$ to $D_i$ is flat and formally smooth. Moreover, $D_i$ is regular. Now we follow the proof of \citep[theorem 4.3.36]{Liu} to prove that $E_i$ is also regular: let $x\in E_i$ and $y=f'(x)$. Let $m=\dim \O_{E_i,x}$ and $n=\dim \O_{D_i,y}$. By flatness we can apply \citep[theorem 4.3.12]{Liu} and find that $\dim \O_{(E_i)_y,x}=m-n$. By \citep[1, 19.6.5]{EGA4}, $(E_i)_y$ is regular at $y$, hence $\O_{(E_i)_y,x}$ is generated by $m-n$ elements $b'_{n+1},\ldots,b'_m$. Each of these is the image of an element $b_i$ in the maximal ideal of $\O_{E_i,x}$. If we let $a_1,\ldots,a_m$ be generators of the maximal ideal of $\O_{D_i,y}$, the $m$ elements $a_1,\ldots,a_n,b_{n+1},\ldots,b_m$ generate $\O_{E_i,x}$, which is therefore regular. 

Finally, as $E_i$ is regular and connected (every connected component of $E_i$ contains the closed point of $T'$), it is irreducible. This shows that the function from the set of components of $D\times_{S'}T'$ to the set of components of $D$ is a bijection and we apply \cref{lemma_bc}.
\end{proof}

\begin{example}
Let $R$ be a regular, noetherian local ring, with a system of regular parameters $t_1,t_2,\ldots,t_n\in R$. Let $S=\Spec R$ and $D=\{t_1\cdot t_2\cdot\ldots\cdot t_n=0\}$. Let $\mathcal C/S$ be a nodal curve, smooth over $U=S\setminus D$. 

Let $a_1,a_2,\ldots,a_n>0$ be integers and let $T$ be the spectrum of the finite locally-free $R$-algebra $$\frac{R[u_1,\ldots,u_n]}{u_1^{a_1}-t_1,\ldots,u_n^{a_n}-t_n}$$

Then $T$ is a regular scheme, $E=\{u_1\cdot u_2\cdot\ldots\cdot u_n\}$ is a normal crossing divisor, and the components of $E$ are in bijection with those of $D$.

By \cref{lemma_bc}, $\mathcal C/S$ is toric-additive if and only if $\mathcal C_T/T$ is toric-additive
\end{example}

\begin{lemma}\label{ta_open}
Toric additivity is an open condition on $S$.
\end{lemma}
\begin{proof}

Suppose that $\mathcal C/S$ is toric-additive at a geometric point $s$. By \cref{TA_smooth} it is enough to show that $\mathcal C/S$ is toric-additive on an \'etale neighbourhood of $s$, since \'etale morphisms are open. We choose an \'etale neighbourhood of finite type $W\ra S$ of $s$ such that $D_W=D\times_SW$ is a strict normal crossing divisor and such that $s$ belongs to all irreducible components $D_1,\ldots, D_n$ of $D_W$. 

Let $t$ be another geometric point of $W$; we want to show that $\mathcal C_W/W$ is toric-additive at $t$. This is true if $t\not\in D_W$, so we may assume without loss of generality that $t$ belongs to $D_1,\ldots,D_m$ for some $1\leq m\leq n$. Let $\zeta$ be a geometric point lying over the generic point of $D_1\cap D_2\cap \ldots\cap D_m$; write $W_{\zeta},W_t,W_s$ for the spectra of the strict henselizations of $W$ at $\zeta,t,s$ respectively. The morphism $W_{\zeta}\ra W$ factors via $W_s$; hence, by \cref{ta_lemmino}, $\mathcal C_W/W$ is toric-additive at $\zeta$. We also have a natural map $W_{\zeta}\ra W_t$. 

Now, we look at the graph $\Gamma_t$ of $\mathcal C_t$: none of its edges are contracted in the graph $\Gamma_{\zeta}$ of $\mathcal C_{\zeta}$. Hence the generization map $H_1(\Gamma_t,\Z)\ra H_1(\Gamma_{\zeta},\Z)$ is an isomorphism. Composing it with the purity map $H_1(\Gamma_{\zeta},\Z)\ra \bigoplus_{i=1}^mH_1(\Gamma_i,\Z)$, we find the purity map for $\mathcal C_t$. As $\mathcal C_{W_{\zeta}}/W_{\zeta}$ is toric-additive, this latter purity map is an isomorphism, hence $\mathcal C_{W_t}/W_t$ is toric-additive, which completes the proof.
\end{proof}

\begin{lemma}\label{TA_blowups}
Let $\pi\colon \mathcal C\ra \mathcal C'$ be a proper birational morphism between nodal curves $\mathcal C/S$ and $\mathcal C'/S$ smooth over $U$. Then $\mathcal C/S$ is toric-additive if and only if $\mathcal C'$ is.  
\end{lemma}

\begin{proof}
As $\mathcal C'_U/U$ is smooth, the restriction of $\pi$ to $U$ is an isomorphism.
By definition, toric additivity of $\mathcal C/S$ depends only on the semi-abelian scheme $\Pic^0_{\mathcal C/S}$. By \citep{deligne}, Th\'eor\`eme pag.132, semi-abelian extensions of $\Pic^0_{\mathcal C_U/U}$ are unique up to unique isomorphism. Hence the morphism $\Pic^0_{\mathcal C'/S}\ra \Pic^0_{\mathcal C/S}$ induced by $\pi$ is an isomorphism, which completes the proof.
\end{proof}

\section{N\'eron models}\label{sectionNM}
\subsection{The definition of N\'eron model}
Let $S$ be any scheme, $U\subset S$ an open and $A/U$ an abelian scheme. 

\begin{definition}\label{defn_NM}
A \textit{N\'eron model} for $A$ over $S$ is a smooth, separated algebraic space \footnote{defined as in \citep{stacks}\href{http://stacks.math.columbia.edu/tag/025Y}{TAG 025Y}.)} $\mathcal N/S$ of finite type, together with an isomorphism $\mathcal N\times_SU\rightarrow A$, satisfying the following universal property: for every smooth morphism of schemes $T\ra S$ and $U$-morphism $f\colon T\times_SU\ra A$, there exists a unique morphism $g\colon T\ra \mathcal N$ such that $g_{|U}=f.$
\end{definition}

It follows immediately from the definition that a N\'eron model is unique up to unique isomorphism; moreover, applying its defining universal property to the morphisms $m\colon A\times_U A\ra A, i\colon A\ra A,$ and $0_A\colon U\ra A$ defining the group structure of $A$, we see that $\mathcal N/S$ inherits from $A$ a unique $S$-group-space structure.

\subsection{Base change properties}
We proceed to analyse how N\'eron models behave under different types of base change. In general, the property of being a N\'eron model is not stable under arbitrary base change. However, we have that:
\begin{lemma}\label{smooth_base_change}
Let $\mathcal N/S$ be a N\'eron model for $A/U$; let $S'\ra S$ be a smooth morphism and $U'=U\times_SS'$. Then the base change $\mathcal N\times_SS'$ is a N\'eron model of $A_{U'}$.
\end{lemma}
\begin{proof}
Let $X\ra S'$ be a smooth scheme with a morphism $f\colon X_{U'}\ra \mathcal A_{U'}$; by composition with the smooth morphism $S'\ra S$ we obtain a smooth scheme $X\ra S$ and a map $X\times_SU\ra A_U$, which extends uniquely to an $S$-morphism $X\ra \mathcal N$. This is the datum of an $S'$-morphism $X\ra \mathcal N\times_SS'$ extending $f$.
\end{proof}

\begin{lemma}\label{fpqclocal}
Let $\mathcal N/S$ be a smooth, separated algebraic space of finite type with an isomorphism $\mathcal N\times_SU\ra A$. Let $S'\ra S$ be a faithfully flat morphism and write $U'=U\times_SS'$. If $\mathcal N\times_SS'$ is a N\'eron model of $A\times_UU'$, then $\mathcal N/S$ is a N\'eron model of $A$.
\end{lemma}

\begin{proof}
We first show that $\mathcal N/S$ satisfies the universal property of N\'eron models when the smooth morphism $T\ra S$ is the identity. So, let $f\colon U\ra A$ be a section of $A/U$. To show that $f$ extends to a section $S\ra \mathcal N$ we only need to check that the schematic closure $X$ of $f(U)$ inside $\mathcal N$ is faithfully flat over $S$: indeed, $X\ra S$ is birational and separated; if it is also flat and surjective it is automatically an isomorphism. Now, by base change of $f$ we get a closed immersion $f'\colon U'\ra A\times_UU'$, which extends to a section $g'\colon S'\ra \mathcal N\times_SS'$ by hypothesis. The schematic image $g'(S')$ is necessarily the schematic closure of $f'(U')$ inside $\mathcal N\times_SS'$; since taking the schematic closure commutes with faithfully flat base change, we have $g'(S')=X\times_SS'$. We deduce that $X\ra S$ is faithfully flat, as its base change via $S'\ra S$ is such. Hence $f\colon U\ra A$ extends to a section $g\colon S\ra \mathcal N$. The uniqueness of the extension is a consequence of the separatedness of $\mathcal N$.

Next, let $T\ra S$ be smooth and let $f\colon T_U\ra A$. In order to extend $f$ to a morphism $g\colon T\ra \mathcal N$, it is enough to show that $\mathcal N\times_ST$ satisfies the extension property for sections $T_U\ra A\times_UT_U$. By the previous paragraph, it is enough to know that $(\mathcal N\times_ST)\times_SS'=(\mathcal N\times_SS')\times_ST$ is a N\'eron model of $(A\times_UT_U)\times_UU'$. This is true by \cref{smooth_base_change}, concluding the proof.
\end{proof}

\begin{lemma}\label{desc_smooth}
Let $A/U$ be abelian, $f\colon S'\ra S$ a smooth surjective morphism, $U'=U\times_SS'$, and $\mathcal N'/S'$ a N\'eron model of $A\times_SS'$.
Then there exists a N\'eron model $\mathcal N/S$ for $A$.
\end{lemma}
\begin{proof}
Write $S'':=S'\times_SS'$, $p_1,p_2\colon S''\ra S'$ for the two projections and $q\colon S''\ra S$ for $f\circ p_1=f\circ p_2$. By \cref{smooth_base_change}, both $p_1^*\mathcal N$ and $p_2^*\mathcal N$ are N\'eron models of $q^*A$. By the uniqueness of N\'eron models, we obtain a descent datum for $\mathcal N'$ along $S'\ra S$. Effectiveness of descent data for algebraic spaces (\citep{stacks}\href{http://stacks.math.columbia.edu/tag/0ADV}{TAG 0ADV}) yields a smooth, separated algebraic space $\mathcal N/S$ of finite type. By \cref{fpqclocal}, this is a N\'eron model for $A/U$. 
\end{proof}

Although N\'eron models are not stable under base change (not even flat), they are preserved by localizations, as we see in the following lemma:
\begin{lemma}\label{NMlocaliz}
Assume $S$ is locally noetherian. Let $s$ be a point (resp. geometric point) of $S$ and $\til S$ the spectrum of the localization (resp. strict henselization) at $s$. Suppose that $\mathcal N/S$ is a N\'eron model for $A/U$. Then $\mathcal N\times_S\til S$ is a N\'eron model for $A\times_U \til U$, where $\til U=\til S\times_S U$.
\end{lemma}

\begin{proof}
Let $\til Y\ra \til S$ be a smooth scheme and $\til f\colon \til Y_{\til U}\ra A_{\til U}$ a morphism. We may assume that $\til Y$ is of finite type over $\til S$, hence of finite presentation. By \citep[3, 8.8.2]{EGA4} there exist an open neighbourhood (resp. \'etale neighbourhood) $S'$ of $s$, a scheme $Y'\ra S'$ restricting to $\til Y$ over $\til S$, and a $(U\times_SS')$-morphism $f'\colon Y'\times_{S'}(U\times_SS')\ra \mathcal N\times_S(U\times_SS')$ restricting to $\til f$ on $\til U$. By \cref{smooth_base_change}, $\mathcal N\times_SS'$ is a N\'eron model of $\mathcal N\times_S(U\times_SS')$, hence we get a unique extension $g'\colon Y'\ra \mathcal N\times_SS'$ of $f'$. The base-change of $g'$ via $\til S\ra S'$ gives us the required unique extension of $\til f$.
\end{proof}

\begin{proposition}
Assume that $S$ is regular. If $\mathcal A/S$ is a an abelian algebraic space, then it is a N\'eron model of its restriction $\mathcal A\times_SU$.
\end{proposition}
\begin{proof}
Using \cref{fpqclocal}, we may assume that $S$ is strictly local and that $\mathcal A/S$ is a scheme. We identify $\mathcal A$ with its double dual $\mathcal A''=\Pic^0_{\mathcal A'/S}$. Now let $T\ra S$ be smooth and $f\colon T_U\ra \mathcal A_U$. Then $f$ corresponds to an element of $A_U(T_U)=\Pic^0_{\mathcal A'/S}(T_U)=\Pic^0(\mathcal A'_{T_U})/\Pic^0(T_U)$. Let $\mathcal L_U$ be an invertible sheaf with fibres of degree $0$ on $\mathcal A'_{T_U}$ mapping to $f$ in $\mathcal A_U(T_U)$. As $\mathcal A'_T$ is regular, $\mathcal L_U$ extends to an invertible sheaf of degree $0$ on $\mathcal A'_T$, which yields a $T$-point of $\mathcal A''=\mathcal A$ extending $f$. The uniqueness of the extension follows from the separatedness of $\mathcal A/S$.
\end{proof}

We conclude the subsection by stating the main theorem about N\'eron models in the case where the base $S$ is of dimension $1$.

\begin{theorem}[\citep{BLR}, 1.4/3]\label{dim1}
Let $S$ be a connected Dedekind scheme with fraction field $K$ and let $A/K$ be an abelian variety. Then there exists a N\'eron model $\mathcal N$ over $S$ for $A/K$. 
\end{theorem}

\subsection{The group of connected components of a N\'eron model}
Since we are mainly interested in the N\'eron model of the jacobian of a smooth curve $\mathcal C_U/U$ extending to a nodal curve $\mathcal C/S$, we will now work under the following hypotheses: $S$ is a regular scheme, $U\subset S$ is an open dense, $A/U$ is an abelian scheme and $\mathcal A/S$ is a semi-abelian scheme with $\mathcal A\times_SU=A$.

The following result is extremely useful:

\begin{lemma} \label{A=N0}
Suppose that $A$ admits a N\'eron model $\mathcal N/S$. Then the canonical morphism $\mathcal A\ra \mathcal N$ coming from the universal property of \cref{defn_NM} is an open immersion, and induces an isomorphism from $\mathcal A$ to the fibrewise-connected component of identity $\mathcal N^0$.
\end{lemma}
\begin{proof}
The fact that $\mathcal A\ra \mathcal N$ is an open immersion follows from \cite[IX, Prop. 3.1.e]{SGA7}. For every point $s\in S$ of codimension $1$, the restriction of $\mathcal N$ to the local ring $\O_{S,s}$ is the N\'eron model of its generic fibre, by \cref{NMlocaliz}. It follows by \citep[XI, 1.15]{raynaud} that the induced morphism $\mathcal A\ra \mathcal N^0$ is an isomorphism.
\end{proof}

We obtain an exact sequence of fppf-sheaves on $S$

\begin{equation}\label{ex_seq_comps} 
0\ra \mathcal A\ra \mathcal N\ra \Phi\ra 0
\end{equation}

where $\Phi$ is an \'etale quasi-finite (in general non-separated) group algebraic space over $S$, the \textit{group of connected components} of $\mathcal N$. 

In \citep[IX, 11]{SGA7} this group is studied in depth. Although only the case where $S$ is the spectrum of a discrete valuation ring is treated, most results carry over to more general bases. In this subsection we describe some results about $\Phi$ which are useful for us.

Let us assume that $S$ is moreover strictly local, with closed point $s$, residue field $k$, and fraction field $K$. We are interested in the finite abelian group $\Phi_s(k)$. 
Let $m$ be an integer; multiplication by $m$ on $\mathcal A$ is a faithfully flat morphism; hence, restricting the exact sequence \ref{ex_seq_comps} to the closed fibre and taking $m$-torsion, we find an exact sequence
$$0\ra \mathcal A_s[m]\ra \mathcal N_s[m]\ra \Phi_s[m]\ra 0.$$
Taking $k$-valued points, we have
$$\Phi_s[m](k)=\frac{\mathcal N_s[m](k)}{\mathcal A_s[m](k)}.$$

The group scheme $\mathcal A[m]/S$ is separated, quasi-finite and flat over an henselian local ring; hence it decomposes canonically into a disjoint union $\mathcal A[m]=\mathcal A[m]^f\sqcup B$, with $\mathcal A[m]^f/S$ a finite flat group scheme, and $B_s=\emptyset$. In particular, we find that $\mathcal A[m]^f_K(K)\subseteq A[m](K)$ is the subgroup of points extending to elements of $\mathcal N(S)$ contained in $\mathcal A(S)$. 

We can rewrite 
$$\Phi_s[m](k)=\frac{\mathcal N_s[m](k)}{\mathcal A_s[m](k)}=\frac{\mathcal N[m](S)}{\mathcal A[m](S)}=\frac{\mathcal A[m]_K(K)}{\mathcal A[m]^f_K(K)}$$
where the second equality is due to the fact that $S$ is henselian, and the third equality is due to the universal property of N\'eron models.

Now let $X$ be the character group of the maximal subtorus $T\subseteq \mathcal A_s$. It is a free abelian group of rank $\mu=\dim T$. We let $\mathcal X$ be the constant group scheme on $S$ with value $X$, and $\mathcal X_m:=\mathcal X\otimes_{\Z}\Z/m\Z$. By \citep[IX, 11.6.7]{SGA7} there is a canonical isomorphism
$$\mathcal X_{m,K} =\frac{\mathcal A[m]_K}{\mathcal A[m]^f_K}.$$
As $\mathcal X/S$ is constant, we have $X\otimes\Z/m\Z=\mathcal X_{m,K}(K)$. 
As the natural map 
$$\frac{\mathcal A[m]_K(K)}{\mathcal A[m]^f_K(K)}\ra \frac{\mathcal A[m]_K}{\mathcal A[m]^f_K}(K)$$
is injective, we have obtained a canonical injective homomorphism
$$\Phi_s[m](k)\hookrightarrow X\otimes\Z/m\Z$$

and taking the colimit over $m\in \Z$, a canonical injective homomorphism of abelian groups
\begin{equation}\label{inclusion_phi}\Phi_s(k)\hookrightarrow X\otimes_{\Z}\Q/\Z.
\end{equation}

Let us return to the case where $\mathcal C/S$ is a nodal curve, smooth over $U$, such that $\Pic^0_{\mathcal C_U/U}$ admits a N\'eron model $\mathcal N/S$. In view of \cref{A=N0}, the fibrewise-connected component of identity of $\mathcal N$ is canonically identified with $\Pic^0_{\mathcal C/S}$, and we have an exact sequence of fppf-shaves on $S$ 
$$0\ra \Pic^0_{\mathcal C/S}\ra \mathcal N \ra \Phi\ra 0$$

We are going first of all to define a generization map for the group of components. Let $\zeta$ be any point of $S$ and let $Z\subseteq S$ be its schematic closure, which is still a strictly local scheme. Let $a\in\Phi_s(s)$ be a component of the closed fibre $\mathcal N_s$. By Hensel's lemma we can extend $a$ to a section $\alpha\in \Phi_Z(Z)$. Because $\Phi_Z/Z$ is \'etale, this extension is actually unique. By restriction we find $\alpha_{\zeta}\in\Phi_{\zeta}(\zeta)$. We have constructed a generization homomorphism of finite abelian groups
$$\Phi_s(s)\ra \Phi_{\zeta}(\zeta), \; a\mapsto \alpha_{\zeta}.$$

Assume now that $S$ has a normal crossing divisor $D=D_1\cup D_2\ldots \cup D_n\subset S$, with $U=S\setminus D$. Let $\zeta_1,\ldots,\zeta_n$ be the generic points of $D_1,\ldots,D_n$. We have seen in \cref{split_sing_sh} that $\mathcal C_{\zeta_i}$ has split singularities and that therefore the maximal torus $T_i$ is split (\cref{lemma_split_torus}). This means that the absolute Galois group $\Gal(\overline{\zeta_i}|\zeta_i)$ acts trivially on the character group $X_i$, and in particular on the subgroup $\Phi(\overline{\zeta_i})$ of $X_i\otimes\Q/\Z$. Hence $\Phi_{\zeta_i}/\zeta_i$ is a constant group scheme, completely determined by its group $\Phi_i=\Phi_{\zeta_i}(\zeta_i)=\Phi_{\zeta_i}(\overline \zeta_i)$ of $\zeta_i$-valued points. 

Putting together the generization homorphisms $\Phi(s)\ra \Phi_i$, we obtain a \textit{purity map} of groups of components:

\begin{equation}\label{purity_phi}
\Phi(s)\ra \Phi_1\oplus \Phi_2\oplus\ldots\oplus \Phi_n. 
\end{equation}

\begin{lemma}\label{injective_comps}
The group homomorphism 
$$\Phi(s)\ra \Phi_1\oplus\Phi_2\oplus\ldots\oplus\Phi_n$$
is injective.
\end{lemma} 

\begin{proof}
The diagram
\begin{center}
\begin{tikzcd}
\Phi(s)\arrow[r]\arrow[d] & \Phi_1\oplus\ldots\oplus\Phi_n\arrow[d] \\
X\otimes\Q/\Z \arrow[r] & (X_1\oplus\ldots\oplus X_n)\otimes\Q/\Z
\end{tikzcd}
\end{center}
where the vertical maps are as in \eqref{inclusion_phi} and the horizontal maps are the purity homomorphisms \ref{purity_phi} and \ref{purity_char}, is commutative. The vertical maps are injective; moreover the purity map $X\ra X_1\oplus\ldots X_n$ is also injective (\cref{injective_purity}), and has torsion-free cokernel (\cref{tf_coker}), hence it is still injective when tensored with $\Q/\Z$. It follows that the top horizontal map is injective.

\end{proof}

\section{Toric additivity as a criterion for existence of N\'eron models}\label{section_finale}

\subsection{Aligned curves}
We start by recalling the definition of aligned nodal curve introduced in \citep{holmes}.
\begin{definition}[\citep{holmes}, definition 2.11]\label{def_aligned}
Let $\mathcal C\ra S$ be a nodal curve and $s$ a geometric point of $S$. Consider the base change $\mathcal C_{S'}/S'$ to the strict henselization $S'$ at $s$, and the labelled graph $(\Gamma,l)$ of \cref{def_labelled_graph}. We say that $\mathcal C/S$ is \textit{aligned at $s$} if for every cycle $\gamma\subset \Gamma$ and every pair of edges $e,e'$ of $\gamma$, there exist integers $n,n'$ such that $$l(e)^n=l(e')^{n'}.$$

We say that $C/S$ is \textit{aligned} if it is aligned at every geometric point of $S$. 
\end{definition}

The following is the main theorem of \citep{holmes}, establishing the relation between alignment and existence of N\'eron models of jacobians: 

\begin{theorem}[\citep{holmes}, theorem 5.16, theorem 6.2] \label{holmes} Let $S$ be regular, $U\subset S$ a dense open, $f\colon \mathcal C\ra S$ a nodal curve, with $f_U\colon \mathcal C_U\ra U$ smooth. 
\begin{itemize}
\item[i)] If the jacobian $\Pic^0_{\mathcal C_U/U}$ admits a N\'eron model over $S$, then $\mathcal C/S$ is aligned;
\item[ii)] if $\mathcal C$ is regular and $\mathcal C/S$ is aligned, then $\Pic^0_{\mathcal C_U/U}$ admits a N\'eron model over $S$.
\end{itemize}

\end{theorem}

\begin{remark} \label{remark_construction_NM}
The proof of \cref{holmes}, part ii) is constructive: for $\mathcal C/S$ aligned with $\mathcal C$ regular, the N\'eron model is the smooth, separated group-algebraic space of finite type
$$\mathcal N=\frac{\Pic^{[0]}_{\mathcal C/S}}{E}$$
where:
\begin{itemize}
\item $\Pic^{[0]}_{\mathcal C/S}$ is the smooth group-algebraic space representing the functor of invertible sheaves on $\mathcal C/S$ with fibres of total degree zero; equivalently, it is the schematic closure of $\Pic^{0}_{\mathcal C_U/U}$ inside $\Pic_{\mathcal C/S}$.
\item $E/S$ is the schematic closure of the unit section $U\ra \Pic^{0}_{\mathcal C_U/U}$ inside $\Pic^{[0]}_{\mathcal C/S}$.
\end{itemize}
The fact that the quotient of fppf-sheaves $\mathcal N$ exists as an algebraic space is due to the fact that under the assumptions of \cref{holmes}, part ii), $E/S$ is a flat subgroup space of $\Pic^{[0]}_{\mathcal C/S}$. See theorem 6.2 of \citep{holmes} for more details.
\end{remark}

In the case where $\mathcal C/S$ is smooth outside of a normal crossing divisor, we have the notion of toric additivity introduced in \cref{section_TA}; we are going to explore its relation with alignment.

Let's consider then a regular base scheme $S$ with a normal crossing divisor $D\subset S$, and a nodal curve $\mathcal C/S$, such that the base change $\mathcal C_U/U$ to $U=S\setminus D$ is smooth.

If $S'\ra S$ is a strict henselization at some geometric point $s$ of $S$, and $D\cap S'$ is given by regular parameters $t_1,\ldots,t_n\in \O(S')$, then the thickness of any non-smooth point $p\in \mathcal C_s$ is generated by $t_1^{m_1}\cdot\ldots\cdot t_n^{m_n}$ for some non-negative integers $m_1,\ldots,m_n$. In particular, $\mathcal C$ is regular at $p$ if and only if its thickness is generated by $t_i$ for some $1\leq i\leq n$.
\begin{lemma}\label{ta_aligned}
Suppose that $\mathcal C/S$ is toric-additive. Then it is aligned. 
\end{lemma}
\begin{proof}
As both alignment and toric additivity are checked over the strict henselizations at geometric points of $S$, we may assume that $S$ is strictly local, with $\{D_i\}_{i=1,\ldots,n}$ the components of the divisor $D$. Each of them is cut out by a regular element $t_i\in \O_S(S)$ and is itself a regular, strictly local scheme. Let $\zeta_i$ be the generic point of $D_i$. By \cref{split_sing_sh}, the curve $\mathcal C_{\zeta_i}$ has split singularities; its labelled graph $(\Gamma_{\zeta_i},l_{\zeta_i})$ is obtained from $(\Gamma_s,l_s)$ by contracting edges according to the procedure in \cref{labelled}; that is, by contracting the edges with label generated by an element of $\O(S)$ invertible at the point $\zeta_i$. 

We have an injective homomorphism \cref{purity_betti}
$$F\colon H_1(\Gamma_s,\Z)\ra H_1(\Gamma_1,\Z)\oplus\ldots\oplus H_1(\Gamma_n,\Z).$$

Notice that, for a general graph $G$, the choice of an orientation of the edges allows to see a cycle on $G$ as a labelling of the edges $E\ra \Z$ taking only values in $\{-1,0,1\}$, and in particular as an element of $H_1(G,\Z)$. Moreover, $H_1(G,\Z)$ is generated by elements arising from cycles of $G$. Therefore we make a choice of orientation on $\Gamma_s$, which determines orientations of the $\Gamma_i$'s as well. 

We only need to show that if $\mathcal C/S$ is not aligned, then $F$ is not surjective.

Consider the commutative diagram with exact rows

\begin{center}
\begin{tikzcd}
0 \arrow[r] & H_1(\Gamma_s,\Z) \arrow[r]\arrow[d, "F=F_1\oplus\ldots\oplus F_n"] & \Z^{E} \arrow[r, "\alpha"]\arrow[d] & \Z^V \arrow[d] \\
0 \arrow[r] & \bigoplus_iH_1(\Gamma,\Z_i) \arrow[r] & \bigoplus_i\Z^{E_i} \arrow[r,"\beta"] & \bigoplus_i \Z^{V_i}.
\end{tikzcd}
\end{center}
with maps as described in \cref{diagram_bho}.
The central vertical map is injective. If we call $A=\ker(\im \alpha \ra\im \beta), B=\coker F, C=\coker(\Z^E\ra\bigoplus_i \Z^{E_i}), D=\coker(\Z^V\ra\bigoplus_i \Z^{V_i})$, we have an exact sequence
$$0\ra A\ra B\ra C\ra D$$
by the Snake Lemma.  

Assume that $\mathcal C/S$ is not aligned. This means that in the labelled graph $(\Gamma_s,l_s)$ over the closed point $s\in S$, there is a cycle $\gamma$ with two edges $e_1,e_2$ such that $l_s(e_1)^m\neq l_s(e_2)^{m'} $ for all $m,m'$ non-negative integers. 

There are two possibilities:
\begin{itemize}
\item[Case 1:] for every edge $e\in \gamma$ there exists an $i$ and an $m$ such that $l_s(e)=t_i^m$. Assume this is the case. Then there are two adjacent edges $e_1,e_2$, such that (without loss of generality) the label of $e_1$ is a power of $t_1$ and the label of $e_2$ is a power of $t_2$. Let's see $\gamma$ as an edge labelling in $H_1(\Gamma_s,\Z)$, and let $\phi\in\Z^E$ be the edge labelling described by 
$$ \phi(e)=
\begin{cases}
\gamma(e) \mbox{ if } l_s(e) \mbox{ is a power of }t_1 \\
0 \mbox{ otherwise.}
\end{cases}
$$
In particular $\phi(e_1)=\pm 1$ and $\phi(e_2)=0$. Let $v\in V$ be a vertex lying on the cycle $\gamma$ and joining $e_1$ and $e_2$. Then $\alpha(\phi)$ is non-zero at $v$. On the other hand, in $\Gamma_1$ all edges of $\gamma$  the label of which is not a power of $t_1$ are contracted; hence the image of $\phi$ via the central vertical map is actually contained in $\bigoplus H_1(\Gamma_i,\Z)$. In particular, by the commutativity of the diagram $\alpha(\phi)$ is sent to zero by $\Z^V\ra \bigoplus \Z^{V_i}$. But then $\alpha(\phi)$ is a non-zero element of $A$, which implies that $B\neq 0$ and that $F$ is not surjective.

\item[Case 2:] there exists some edge $\overline e$ in $\gamma$ that is preserved in at least two of the contracted graphs $\Gamma_i$'s. We may assume these are $\Gamma_1$ and $\Gamma_2$. In this case, we define an element $\phi=(\phi_1,\ldots,\phi_n)\in\bigoplus H_1(\Gamma_i,\Z)$ as follows: $\phi_1$ is the (non-trivial) cycle image of $\gamma$ in $\Gamma_1$; and for $2\leq i\leq n$, $\phi_i=0$. Hence $\phi$ is a non-zero element of $\bigoplus H_1(\Gamma_i,\Z)$. Now let $\psi$ be any element of $H_1(\Gamma_s,\Z)$; if $\psi(\overline e)=0$ then $F_1(\psi)\neq \phi_1$, and if $\psi(\overline e)\neq 0$, then $F_2(\psi)\neq \phi_2$. Therefore $\phi$ is not in the image of $F$, and the proof is finished.
\end{itemize} 

\end{proof}

\begin{lemma} \label{regular-aligned-TA}
Suppose that the total space $\mathcal C$ is regular.
Then $\mathcal C/S$ is aligned if and only if $\mathcal C/S$ is toric-additive.
\end{lemma}
\begin{proof}
Once again we may assume that $S$ is strictly local. Let $\Gamma_s=(V_s,E_s)$ be the dual graph of the fibre of $\mathcal C$ over the closed point $s\in S$, and $l_s\colon E_s\rightarrow L_s$ the labelling of the edges, taking value in the monoid $L_s$ of principal ideals of $\O_S(S)$. Since $\mathcal C$ is regular, the labels can only take the values $(t_1),\ldots,(t_n)\in L_s$. This means that $\mathcal C/S$ is aligned if and only if every cycle of $\Gamma$ has edges with the same label.

Now, let $\{D_i\}_{i=1,\ldots,n}$ be the components of the divisor $D$. Each of them is cut out by a regular element $t_i\in \O_S(S)$ and is itself a regular, strictly local scheme. Let $\zeta_i$ be the generic point of $D_i$. By \cref{split_sing_sh}, the curve $\mathcal C_{\zeta_i}$ has split singularities; its labelled graph $(\Gamma_{\zeta_i},l_{\zeta_i})$ is obtained from $(\Gamma_s,l_s)$ by contracting edges according to the procedure in \cref{labelled}; that is, by contracting the edges with label different from $(t_i)$. 

We have already proved in \cref{ta_aligned} that toric additivity implies alignment.
Assume that $\mathcal C/S$ is aligned; we want to show that the injective homomorphism 
$$F\colon H_1(\Gamma_s,\Z)\ra H_1(\Gamma_1,\Z)\oplus\ldots\oplus H_1(\Gamma_n,\Z)$$
is surjective.

We choose an orientation of the edges of $\Gamma_s$, so we can associate to every cycle $\gamma$ in $\Gamma_s$ an element $\gamma\in H_1(\Gamma_s,\Z)$.

Fix $1\leq i \leq n$, let $\gamma$ be a cycle in $\Gamma_i$, and $(0,0,\ldots,0,\gamma,0,\ldots,0)\in \bigoplus_iH_1(\Gamma_i,\Z)$ the corresponding element. The edges of $\gamma$, when seen in $\Gamma_s$, have label $(t_i)$. Moreover, they are adjacent, by the hypothesis of  alignment. Hence they still form a cycle in $\Gamma_s$ which is sent to $(0,\ldots,0,\gamma,0,\ldots,0)$ via the map $F$. We conclude that $F$ is surjective.

\end{proof}

\subsection{Toric additivity and desingularization of curves}
Let $S$ be a regular base scheme $S$ with a normal crossing divisor $D\subset S$, and let $\mathcal C/S$ be a nodal curve, such that the base change $\mathcal C_U/U:=S\setminus D$ is smooth. 

In \citep[3.6]{alterations}, it is proven that if $\mathcal C\ra S$ has \textit{split} fibres, there exists a blow-up $\phi\colon \mathcal C'\ra \mathcal C$ such that $\mathcal C'\ra S$ is still a nodal curve, and $\mathcal C'$ is \textit{regular}. The condition of splitness implies that the irreducible components of the geometric fibres are smooth; or equivalently, that the dual graphs of the geometric fibres do not admit loops. We are going to introduce a condition on $\mathcal C/S$, weaker than splitness, and show that a statement analogous to the one in \citep[3.6]{alterations} holds for curves satisfying this condition.

\begin{definition}\label{def_disciplined}
Let $\mathcal C\ra S$ be a nodal curve. We say that $\mathcal C/S$ is \textit{disciplined} if, for every geometric point $\overline s$ of $S$, and $p\in \mathcal C^{ns}_{\overline s}$, at least one of the following is satisfied:
\begin{itemize}
\item[i)] $p$ belongs to two irreducible components of $\mathcal C_{\overline s}$;
\item[ii)] the thickness of $p$ is generated by a power of a regular parameter of $\O_{S,\overline s}^{sh}$.
\end{itemize}
\end{definition}

\begin{example}\label{ex_not_disc}
Let $k$ be a field, $S=\Spec k[[u,v]]$, $U=S\setminus \{uv=0\}$. Let $\mathcal E/S$ be the family of nodal curves of arithmetic genus one 
$$y^2-x^3-x^2-uv=0$$
smooth over $U$.

The family $\mathcal E/S$ is not disciplined: the point $p=(x=0,y=0,u=0,v=0)$ belongs to only one component of the fibre over $u=0,v=0$, and its thickness is the ideal $(uv)$.
\end{example}

We give first an auxiliary lemma:
\begin{lemma}\label{disc_graph}
Hypothesis as in the beginning of the subsection. Suppose that $S$ is strictly local; write $D=D_1\cup \ldots \cup D_n$ and write $\zeta_i$ for the generic point of $D_i$. Suppose that $\mathcal C/S$ is disciplined. Let $p\in \mathcal C_s^{ns}$ be a non-smooth point of the fibre over the closed point, such that $p$ does not satisfy condition ii) of \cref{def_disciplined}. Let $X_1,X_2$ be the distinct irreducible components of the closed fibre $\mathcal C_s$ containing $p$. Then there exists $i\in\{1,\ldots,n\}$ and $Y_1,Y_2$ irreducible components of $\mathcal C_{\zeta_i}$ whose schematic closures $\overline Y_1,\overline Y_2\subset \mathcal C$ satisfy:  $X_1\not\subset \overline Y_2\supset X_2$ and $X_2\not\subset \overline Y_1\supset X_1$. 
\end{lemma}
\begin{proof}

Let $(\Gamma_s,l_s)$ be the labelled graph of $\mathcal C_s$. By hypothesis, the edge $e(p)$ corresponding to $p$ has distinct extremal vertices, $v_1$ and $v_2$, and label $t_1^{m_1}\cdot\ldots\cdot t_l^{m_l}$, with $2\leq l\leq n$ and $m_1,\ldots,m_l\geq 1$. The fibres over the generic points $\zeta_1,\ldots,\zeta_n$ have split singularities by \cref{split_sing_sh}, so we can consider their labelled graphs $(\Gamma_i,l_i)$. What we want to prove is that there exists $i\in\{1,\ldots,l\}$ such that $v_1$ and $v_2$ are mapped to distinct vertices of $(\Gamma_i,l_i)$ via the procedure described in \cref{labelled}. 

Suppose the contrary. As $e(p)$ is not contracted in any $\Gamma_i$, $1\leq i\leq l$, we deduce that there exists a cycle $\gamma$ in $\Gamma_s$, containing $e(p)$, such that for all $1\leq i\leq l$, all edges $e\neq e(p)$ of $\gamma$ are contracted in $\Gamma_i$. Let $\zeta_{12}$ be the generic point of $D_1\cap D_2$; all edges $e\neq e(p)$ of $\gamma$ are contracted in $\Gamma_{12}$, the labelled graph of $\mathcal C_{\zeta_{12}}$, and in particular $v_1$ and $v_2$ are mapped to the same vertex. The edge $e(p)$ is therefore mapped to a loop in $\Gamma_{12}$, with label $t_1^{m_1}t_2^{m_2}$. However, this contradicts the fact that $\mathcal C\ra S$ is disciplined at $\zeta_{12}$, and we have obtained a contradiction.
\end{proof}

We introduce now some notation: given a scheme $X$, we will denote by Sing$(X)\subseteq X$ the set of points that are not regular. We say that the \textit{center} of a blow-up $\pi\colon Y\ra X$ is the complement of the largest open $U\subset X$ such that $\pi^{-1}(U)\ra U$ is an isomorphism.

\begin{lemma}\label{disc->reg}
Hypotheses as in the beginning of the subsection.
Suppose $f\colon \mathcal C\ra S$ is disciplined. Then there is an \'etale surjective $g\colon S'\ra S$ and a blow-up $\phi\colon \mathcal C' \ra \mathcal C\times_SS'$ such that 
\begin{itemize}
\item the center of $\phi$ is contained in Sing$(\mathcal C\times_SS')$;
\item $\mathcal C'$ is a nodal curve over $S'$, smooth over $g^{-1}(U)$;
\item $\mathcal C'$ is regular.
\end{itemize}
\end{lemma}

\begin{proof}
First, notice that the order in which the blow-ups of the curve and the \'etale covers of the base are taken does not matter, as blowing-up commutes with \'etale base change.
After replacing $S$ by a suitable \'etale cover, we may assume that $D$ is a strict normal crossing divisor. We can now apply \citep[3.3.2]{alterations} and assume that Sing$(\mathcal C)\subset \mathcal C$ has codimension at least $3$. 

Next, we claim that there exists an \'etale cover $S'\ra S$, such that for every point $s'\in S'$, the irreducible components of $\mathcal C_{s'}$ are geometrically irreducible. To prove the claim, let's take $s\in S$. Replacing $S$ by an \'etale neighbourhood of $s$ in $S$ we may assume that $\mathcal C/S$ admits sections $\sigma_1,\ldots,\sigma_r$ through the smooth locus of $\mathcal C/S$ and intersecting every irreducible component of the fibre of $\mathcal C_s$. Now, the sheaf $\mathcal F:=\mathcal O(\sigma_1+\ldots+\sigma_r)$ is ample over $s$. Since ampleness is an open condition there exists $U\subset S$ open neighbourhood of $s$ where $\mathcal F$ is ample. Then for every point $u\in U$, every irreducible component of $\mathcal C_u$ is met by a section $\sigma_i$, and is therefore geometrically irreducible. This proves the claim.

Hence, replacing $S$ by a suitable \'etale cover, we may assume that for every generic point $\zeta$ of $D$, the fibre $\mathcal C_{\zeta}$ has irreducible components that are geometrically irreducible.


 Now, let $E$ be an irreducible component of $\mathcal C_{D}=\mathcal C\times_SD$ and let $\pi\colon \mathcal C'\ra \mathcal C$ be the blow-up of $\mathcal C$ along $E$. If $p\in E$ is a regular point of $\mathcal C$, $f$ is an isomorphism at $p$, because $E$ is cut out by one equation. Otherwise, the completion of the strict henselization at (a geometric point lying over) $p$ is of the form 
$$\widehat{\O}_{\mathcal C,\overline p}^{sh}\cong\frac{\widehat{\O}_{S,\overline{f(p)}}^{sh}[[x,y]]}{xy-t_1^{m_1}\cdot\ldots \cdot t_l^{m_l}}$$
with $t_1,\ldots,t_n$ regular parameters cutting out $D$, $1\leq l\leq n$ and positive integers $m_1,\ldots,m_l$. In fact, because the singular locus has codimension at least three, we have $l\geq 2$, and $m_1=\ldots=m_l=1$. 

The ideal of the pullback of $E$ to $\widehat{\O}_{\mathcal C,\overline p}^{sh}$ is either $(t_i)$ for some $1\leq i\leq l$, or one between $(x,t_i)$ and $(y,t_i)$ for some $1\leq i \leq l$. In the first case, $\pi$ is an isomorphism at $p$. In the second case, one can compute explicitly the blowing up of $\Spec \O^{sh}_{\mathcal C,\overline p}$ at the ideal $(x,t_i)$ (or $(y,t_i)$) and find that $f'\colon \mathcal C'\ra S$ is still a nodal curve, disciplined, with Sing$(\mathcal C)$ of codimension at least three, and such that for every generic point $\zeta$ of $D$ the fibre $\mathcal C_{\zeta}$ has irreducible components that are geometrically irreducible. We omit the explicit computations. 

 Let $Y\subset \mathcal C$ be the center of $\pi\colon \mathcal C'\ra \mathcal C$. Then $Y$ consists only of non-regular points, hence it has codimension at least $3$. As $f\colon \mathcal C'\ra S$ is a curve, the fibres of $\pi$ have dimension at most $1$, hence $\pi^{-1}(Y)$ has codimension at least $2$ in $\mathcal C$. It follows that there is a bijection between the irreducible components of $\mathcal C_{D}$ and $\mathcal C'_{D}$, given by taking the preimage under $\pi$. Now, $\pi^{-1}(E)$ is a divisor, and for any other irreducible component $E'$ of $\mathcal C_D$ that is a divisor, $\pi^{-1}(E')$ is also a divisor. We conclude that $\pi^*\colon \mathcal C^*\ra \mathcal C$, the composition of the blowing-ups of all irreducible component of $\mathcal C_{D}$, is such that every component of $\mathcal C^*_D$ is a divisor. 
  Besides, as previously noticed, $f^*\colon \mathcal C^*\ra S$ is a nodal curve, disciplined, and Sing$(\mathcal C^*)$ has codimension at least three.

Assume now by contradiction that Sing$(\mathcal C^*)\neq \emptyset$, and let $p\in $ Sing$(\mathcal C^*)$. Then without loss of generality the thickness at $p$ is $(t_1\cdot\ldots\cdot t_l)$ for some $2\leq l\leq n$. Consider the base change $\mathcal C^*_T/T$, where $T$ is the spectrum of some strict henselization at $s=f^*(p)$. For every $i$ let $\xi_i$ be the generic point of $D_i\cap T$. By \cref{disc_graph}, for some $i\in\{1,\ldots,l\}$, there are distinct components $Y_1,Y_2$ of $\mathcal C^*_{\xi_i}$ whose closure in $\mathcal C^*_{T\cap D_i}$ contain $p$. Because the irreducible components of the fibre $ \mathcal C^*_{\zeta_i}$ are geometrically irreducible, there are components $X_1,X_2$ of $\mathcal C^*_{\zeta_i}$ whose closures $E_1,E_2$ in $\mathcal C^*_{D_i}$ contain $p$. But then $E_1$ and $E_2$ are given by $(x,t_1)$ and $(y,t_1)$ in $\widehat{\O}_{\mathcal C^*,\overline p}^{sh}$. In particular, they are not divisors. This is a contradiction, and therefore Sing$(\mathcal C^*)=\emptyset$.
\end{proof}

\begin{lemma}\label{TA->disc}
Hypotheses as in the beginning of the subsection. Suppose that $f\colon \mathcal C \ra S$ is toric-additive. Then $\mathcal C/S$ is disciplined.
\end{lemma}
\begin{proof}
We may assume that $S$ is strictly local, with closed point $s$, and with $D$ given by a system of regular parameters $t_1\ldots, t_n$. Let $p\in \mathcal C^{ns}_s$, with thickness $t_1^{m_1}\cdot\ldots\cdot t_l^{m_l}$ for some $1\leq l \leq n$ and $m_1,\ldots,m_l\geq 1$. We have to show that if $l\geq 2$ then $p$ lies on two components of $\mathcal C_s$ . 

Suppose by contradiction that $l\geq 2$ and that $p$ lies on only one component of $\mathcal C_s$. The dual graph $\Gamma$ over $s$ has a loop $L$ corresponding to $p$, with label $t_1^{m_1}\cdot\ldots\cdot t_l^{m_l}$. For $1\leq i \leq n$ call $\Gamma_i$ the dual graph of the fibre $\mathcal C_{\zeta_i}$ over the generic point of $D_i$. The loop $L$ is preserved in the dual graphs $\Gamma_i$ for $1\leq i\leq l$. Let $\Gamma'$ be the graph obtained by $\Gamma$ by removing the loop $L$, and define similarly $\Gamma'_i$, $1\leq i \leq l$. Now inequality \ref{ineq_betti} of \cref{inequality} says
$$h_1(\Gamma',\Z)\leq \sum_{i=1}^lh_1(\Gamma'_i,\Z)+\sum_{j=l+1}^nh_1(\Gamma_j,\Z).$$

For every $1\leq i \leq l$, $h_1(\Gamma_i,\Z)=h_1(\Gamma'_i,\Z)+1$. Since $l\geq 2$, we find that 
$h_1(\Gamma,\Z)=h_1(\Gamma',\Z)+1<\sum_{i=1}^nh_1(\Gamma_i,\Z)$. This contradicts the fact that $\mathcal C/S$ is toric-additive.
\end{proof}

\subsection{Toric additivity and existence of N\'eron models}
We consider again a regular base scheme $S$ with a normal crossing divisor $D\subset S$, and a nodal curve $\mathcal C/S$ such that the base change $\mathcal C_U/U:=S\setminus D$ is smooth.
\Cref{holmes} tells us that if $\Pic^0_{\mathcal C_U/U}$ admits a N\'eron model over $S$, then $\mathcal C/S$ is aligned. We show that being disciplined is also a necessary condition for existence of a N\'eron model. 

\begin{lemma}\label{NM->disc}
Assume that $S$ is an excellent scheme. Suppose that $\mathcal C/S$ is such that $\Pic^0_{\mathcal C_U/U}$ admits a N\'eron model $\mathcal N$ over $S$. Then $\mathcal C/S$ is disciplined.
\end{lemma}

\begin{proof}
We may assume that $S$ is strictly henselian, with closed point $s$ and residue field $k=k(s)$. 
Assume by contradiction that $\mathcal C/S$ is not disciplined. Then there is some $p\in \mathcal C^{ns}_s$ that belongs to only one component $X$ of $\mathcal C_s$, and such that its thickness is $t_1^{m_1}\cdot\ldots\cdot t_l^{m_l}$ with $m_i\geq 1$ and $2
\leq l \leq n$. Let $q\in \mathcal C_s(k)$ be a smooth $k$-rational point belonging to the same component as $p$. By Hensel's lemma, there exists a section $\sigma_q\colon S \ra \mathcal C$ through $q$. We claim that the same is true for $p$: let $\widehat S$ be the spectrum of the completion of $\O(S)$ at its maximal ideal and consider the morphism $$W:=\Spec \widehat{\O}^{sh}_{\mathcal C,p}\cong \Spec \frac{\O(\widehat{S})[[x,y]]}{xy-t_1^{m_1}\cdot\ldots\cdot t_l^{m_l}} \ra \widehat S.$$

This has a section given by $x=t_1^{m_1}$, $y=t_2^{m_2}\cdot\ldots\cdot t_l^{m_l}$. Composing the section with the canonical morphism $W\ra \mathcal C$, gives a morphism $\widehat \sigma_p\colon \widehat S\ra \mathcal C$ going through $p$. Because $S$ is excellent and henselian, it has the Artin approximation property, and there exists a section $\sigma_p\colon S\ra \mathcal C$ which agrees with $\widehat \sigma_p$ when restricted to the closed point $s$, hence going through $p$.

 We write $\mathcal F:=\mathcal I(\sigma_p)\otimes_{\mathcal O_{\mathcal C}}\mathcal O(\sigma_q)$ for the coherent sheaf on $\mathcal C$ given by the tensor product of the ideal sheaf of $\sigma_p$ with the invertible sheaf associated to the divisor $\sigma_q$. It is what is called a \textit{torsion free, rank $1$} sheaf in the literature: it is $S$-flat, its fibres are 
of rank $1$ at the generic points of fibres of $\mathcal C$, and have no embedded points. Notice that $\mathcal F$ is not an invertible sheaf, as $\dim_{k(p)}\mathcal F\otimes k(p)=2.$

Let $u_p$ and $u_q$ be the restrictions of $\sigma_p$ and $\sigma_q$ to $U$. They are $U$-points of the smooth curve $\mathcal C_U/U$; the restriction of $\mathcal F$ to $U$ is the invertible sheaf $\mathcal F_U=\O_{\mathcal C_U}(u_q-u_p)$. This is the datum of a $U$-point $\alpha$ of $\Pic^0_{\mathcal C_U/U}$: indeed, $\Pic(U)=0$ because $\O(U)$ is a UFD, and $C_U/U$ has a section, so $\Pic^0_{\mathcal C_U/U}(U)=\Pic^0(\mathcal C_U)$.

By the definition of N\'eron model, there is a unique section $\beta\colon S\ra \mathcal N$ with $\beta_U=\alpha$. We write $J$ for $\Pic^0_{\mathcal C/S}$. As $J$ is semi-abelian, the canonical open immersion $J\ra \mathcal N$ identifies $J$ with the fibrewise-connected component of identity $\mathcal N^0$ (\cref{A=N0}). Write $\zeta_i, i=1\ldots, n$ for the generic points of the divisors $D_i$. Then $S_i:=\Spec \O_{S,\zeta_i}$ is a trait, and the restriction $\mathcal N_{S_i}$ is a N\'eron model of its generic fibre. Therefore $\alpha_K$ extends uniquely to a section $\alpha_i\colon S_i\ra \mathcal N_{S_i}$. As $\mathcal F_{S_i}$ is an invertible sheaf of degree $0$ on every irreducible component of $\mathcal C_{\overline \zeta_i}$, $\mathcal F_{S_i}$ is a $S_i$-point of $J_{S_i}$, and $\alpha_i$ is given by $\mathcal F_{S_i}$. Therefore, the restriction of $\alpha\colon S\ra \mathcal N$ to $S_i$ factors through $J=\mathcal N^0$ for every $i=1\ldots,n$.

 We denote now by $\Phi/S$ the \'etale group space of connected components of $\mathcal N$. By \cref{injective_comps}, the canonical morphism
$$\Phi(s)\ra \bigoplus_{i=1}^n\Phi(\zeta_i)$$
is injective. 
This implies that $\alpha$ lands inside $J=\mathcal N^0$, or in other words that $\mathcal F_U$ extends to an invertible sheaf $\mathcal L$ on $\mathcal C$ such that $\mathcal L_s$ is of degree $0$ on every component.

Now, let $Z\ra S$ be a closed immersion, with $Z$ a trait, such that the generic point $\xi$ of $Z$ lands into $U$ (it is an easy check that such a closed immersion exists). As $\mathcal F_{\xi}$ and $\mathcal L_{\xi}$ define the same point of $\Pic^0_{\mathcal C_{\xi}/\xi}$, there are isomorphisms $\mu_{\xi}\colon \mathcal F_{\xi}\ra \mathcal L_{\xi}$ and $\lambda_{\xi}\colon \mathcal L_{\xi}\ra \mathcal F_{\xi}$. By the same argument as in \citep[7.8]{alt_kl}, $\mu_{\xi}$ and $\lambda_{\xi}$ extend to morphisms $\mu\colon \mathcal F_Z\ra \mathcal L_Z$ and $\lambda\colon \mathcal L_Z\ra \mathcal F_Z$, which are non-zero on all fibres. Let's look at the restrictions to the closed fibre, $\mu_s\colon \mathcal F_s\ra \mathcal L_s$, $\lambda_s\colon\mathcal L_s\ra \mathcal F_s$. We know that $\mathcal F_s$ is trivial away from the component $X\subset \mathcal C_s$. So, if we write $Y$ for the closure in $\mathcal C_s$ of the complement of $X$, we may restrict $\mu_s$ and $\lambda_s$ to $Y$ to get global sections $l$ and $l'$ of $\mathcal L_Y$ and $\mathcal L^{\vee}_Y$ respectively. Now, if $l=0$, then the restriction $\mu_X$ of $\mu_s$ to $X$ is non-zero, because $\mu_s$ is non-zero. If $l\neq 0$, as $\mathcal L_s$ is of degree zero on every component, we have $l(y)\not\in\m_y\mathcal L_y$ for every $y\in Y$, and in particular for $y\in Y\cap X$. It follows that also in this case $\mu_X\neq 0$. We can apply the same argument to $l'$ and conclude that $\lambda_X\neq 0$. Then the compositions $\mu_X\circ\lambda_X\colon \mathcal L_X\ra \mathcal L_X$ and $\lambda_X\circ \mu_X\colon \mathcal F_X\ra \mathcal F_X$ are non-zero. As $\End_{\O_X}(\mathcal F_X)=k=\End_{\O_X}(\mathcal L_X)$, they are actually isomorphisms. It follows that $\mu_X\colon \mathcal F_X\ra \mathcal L_X$ is an isomorphism. However, $\dim_{k(p)}\mathcal F_{k(p)}=2$, while $\mathcal L_X$ is an invertible sheaf. This gives us the required contradiction.

\end{proof}

\begin{example}
Consider again the curve $\mathcal E/S$ of \cref{ex_not_disc}. The fibres of $\mathcal E/S$ are geometrically irreducible, hence the family is aligned. As the total space $E$ is not regular, \cref{holmes} does not allow us to deduce the existence of a N\'eron model for $\mathcal E_U=\Pic^0_{\mathcal E_U/U}$ over $S$.

However we have seen that $\mathcal E/S$ is not disciplined; hence by \cref{NM->disc} we know for certain that there exists no N\'eron model over $S$ for $\mathcal E_U$.
\end{example}

Combining the previous lemmas of this section, we obtain the following theorem, which shows that toric additivity is a criterion for existence of N\'eron models of jacobians:

\begin{theorem}\label{main_thm_curves}
Let $S$ be a regular scheme, $D$ a normal crossing divisor on $S$, $\mathcal C\ra S$ a nodal curve smooth over $U=S\setminus D$. 
\begin{itemize}
\item[i)] If $\mathcal C/S$ is toric-additive, then $\Pic^0_{\mathcal C_U/U}$ admits a N\'eron model over $S$. 
\item[ii)] If moreover $S$ is excellent, the converse is also true. 
\end{itemize}
\end{theorem}

\begin{proof}
Whether we are in the hypotheses of i) and ii), we know by \cref{TA->disc,NM->disc} above that $\mathcal C/S$ is disciplined; hence by \cref{disc->reg} there exists an \'etale cover $g\colon S'\ra S$ and a blow-up $\pi\colon \mathcal C'\ra \mathcal C_{S'}$ which restricts to an isomorphism over $U'=U\times_SS'$, such that $\mathcal C'$ is a regular nodal curve.

Assume that $\Pic^0_{\mathcal C/S}$ is toric-additive. To show the existence of a N\'eron model over $S$, it is enough to show it over $S'$, by \cref{desc_smooth}. The base change $\mathcal C_{S'}/S'$ is toric-additive by \cref{TA_smooth}. The blow-up $\mathcal C'/S'$ is also toric-additive by \cref{TA_blowups}. We can now apply \cref{regular-aligned-TA} and deduce that $\mathcal C'/S'$ is aligned. Hence by \cref{holmes}, we find that $\Pic^0_{C_{U'}/U'}$ admits a N\'eron model over $S'$, proving i).

Now assume that $S$ is excellent and that $\Pic^0_{\mathcal C_U/U}$ admits a N\'eron model $\mathcal N$ over $S$. Then $\mathcal N'=\mathcal N\times_SS'$ is a N\'eron model for $\Pic^0_{\mathcal {C'}_{U'}/U'}$ over $S'$, by \cref{smooth_base_change}. Hence $\mathcal C'/S'$ is aligned by \cref{holmes}. As $\mathcal C'$ is regular, we deduce by \cref{regular-aligned-TA} that $\mathcal C'/S'$ is toric-additive. By \cref{TA_blowups}, so is $\mathcal C_{S'}/S'$. As toric additivity descends along \'etale covers (\cref{TA_smooth}), $\C/S$ is toric-additive.
\end{proof}

\begin{corollary}
Let $S$ be an excellent, regular scheme, $D$ a codimension one regular subscheme of $S$. Let $\mathcal C/S$ be a nodal curve, smooth over $U=S\setminus D$. Then $\Pic^0_{\mathcal C_U/U}$ admits a N\'eron model over $S$.
\end{corollary}
\begin{proof}
At the strict henselization of each geometric point $s$ of $S$, $D$ is irreducible. Hence by \cref{remark_ta}, $\mathcal C/S$ is toric-additive at $s$.
\end{proof}

\begin{corollary}
Let $S$ be an excellent, regular scheme, $D$ a normal crossing divisor on $S$, $\mathcal C\ra S$ a nodal curve smooth over $U=S\setminus D$. There exists a biggest open $V\subset S$ over which $\Pic^0_{\mathcal C_U/U}$ admits a N\'eron model. 
\end{corollary}
\begin{proof}
By \cref{ta_open}, toric additivity is an open condition on $S$.
\end{proof}


\begin{thebibliography}{{Sta}16}

\bibitem[AK80]{alt_kl}
Allen~B. Altman and Steven~L. Kleiman.
\newblock Compactifying the {P}icard scheme.
\newblock {\em Adv. in Math.}, 35(1):50--112, 1980.

\bibitem[BLR90]{BLR}
Siegfried Bosch, Werner L{\"u}tkebohmert, and Michel Raynaud.
\newblock {\em N\'eron Models}, volume~21 of {\em Ergebnisse der {M}athematik
  und ihrer {G}renzgebiete}.
\newblock Springer-Verlag, 1990.

\bibitem[Del85]{deligne}
Pierre Deligne.
\newblock Le lemme de {G}abber.
\newblock {\em Ast\'erisque}, 127:131--150, 1985.
\newblock Seminar on arithmetic bundles: the Mordell conjecture (Paris,
  1983/84).

\bibitem[dJ96]{alterations}
A.~J. de~Jong.
\newblock Smoothness, semi-stability and alterations.
\newblock {\em Inst. Hautes \'Etudes Sci. Publ. Math.}, (83):51--93, 1996.

\bibitem[FC90]{faltings1990degeneration}
G.~Faltings and C.L. Chai.
\newblock {\em Degeneration of Abelian Varieties}.
\newblock A Series of modem surveys in mathematics. Springer-Verlag, 1990.

\bibitem[GD67]{EGA4}
Alexander Grothendieck and Jean Dieudonn{\'e}.
\newblock {\em {\'E}l{\'e}ments de g{\'e}om{\'e}trie alg{\'e}brique {IV}},
  volume 20, 24, 28, 32 of {\em Publications {M}ath{\'e}matiques}.
\newblock Institute des {H}autes {\'E}tudes {S}cientifiques., 1964-1967.

\bibitem[GRR72]{SGA7}
Alexander Grothendieck, Michel Raynaud, and Dock~Sang Rim.
\newblock {\em Groupes de monodromie en g\'eom\'etrie alg\'ebrique. {I}}.
\newblock Lecture Notes in Mathematics, Vol. 288. Springer-Verlag, 1972.
\newblock S{\'e}minaire de G{\'e}om{\'e}trie Alg{\'e}brique du Bois-Marie
  1967--1969 (SGA 7 I).

\bibitem[Hol17]{holmes}
David Holmes.
\newblock N{\'e}ron models of jacobians over base schemes of dimension greater
  than 1.
\newblock {\em To appear in Journal f{\"u}r die reine und angewandte
  Mathematik}, 2017.

\bibitem[Liu02]{Liu}
Qing Liu.
\newblock {\em Algebraic geometry and arithmetic curves}, volume~6 of {\em
  Oxford Graduate Texts in Mathematics}.
\newblock Oxford University Press, Oxford, 2002.
\newblock Translated from the French by Reinie Ern{\'e}, Oxford Science
  Publications.

\bibitem[Ray70]{raynaud}
Michel Raynaud.
\newblock {\em Faisceaux amples sur les sch\'emas en groupes et les espaces
  homog\`enes}.
\newblock Lecture Notes in Mathematics, Vol. 119. Springer-Verlag, Berlin-New
  York, 1970.

\bibitem[{Sta}16]{stacks}
The {Stacks Project Authors}.
\newblock \itshape {S}tacks {P}roject.
\newblock \url{http://stacks.math.columbia.edu}, 2016.

\end{thebibliography}
\end{document}